\documentclass[referee,envcountsect,]{svjour3}
\smartqed
\usepackage{graphicx}
\usepackage{xcolor}

\usepackage{amsmath,amsfonts,amssymb}
\usepackage{tikz}
\usepackage{mdframed}
\newmdenv[linecolor=black,skipabove=\topsep,skipbelow=\topsep,
leftmargin=-5pt,rightmargin=-5pt,
innerleftmargin=5pt,innerrightmargin=5pt]{mybox}

\def\R{\mathbb R}
\def\oR{\overline{\mathbb R}}
\def\N{\mathbb N}
\def\co{\mbox{\rm co}\,}
\def\ox{\bar{x}}

\def\co{\mbox{\rm co}\,}

\def\ri{\mbox{\rm ri}\,}

\def\epi{\mbox{\rm epi}\,}

\def\dom{\mbox{\rm dom}\,}

\def\R{\mathbb{R}}
\def\N{\mathbb{N}}

\begin{document}
	
	\title{Optimization Problems with Nearly Convex Objective Functions and Nearly Convex Constraint Sets}
	\titlerunning{Optimization Problems with Nearly Convex Objective Functions...}
	
	
	
	\author{Nguyen Nang Thieu\and Nguyen Dong Yen}
	
	\institute{Nguyen Nang Thieu (Corresponding author) \at
		Institute of Mathematics, Vietnam Academy of Science and Technology, 18 Hoang Quoc Viet, Hanoi 10072, Vietnam\\ nnthieu@math.ac.vn
		\and
		Nguyen Dong Yen \at Institute of Mathematics, Vietnam Academy of Science and Technology, 18 Hoang Quoc Viet,
		Hanoi 10072, Vietnam\\
		ndyen@math.ac.vn}
	
	\date{Received: date / Accepted: date}
	
	\maketitle
	
	\begin{abstract} To every nearly convex optimization problem, that is a minimization problem with a nearly convex objective function and a nearly convex constraint set, we associate a uniquely defined convex optimization problem with a lower semicontinuous objective function and a closed constraint set. Interesting relationships between the original nearly convex problem and the associated convex problem are established. Optimality conditions in the form of Fermat's rules are obtained for both problems. We then get a Lagrange multiplier rule for a nearly convex optimization problem under a geometrical constraint and functional constraints from the Kuhn-Tucker conditions for the associated convex optimization problem. The obtained results are illustrated by concrete examples.
	\end{abstract}
	
	

 	\section{Introduction}\label{sect1}

 The concept of a \textit{nearly convex set} originated from the work of Minty~\cite{Minty1961}. Rockafellar~\cite{r,R1970} showed that the effective domain of the subdifferential mapping of a proper, lower semicontinuous convex function is only a nearly convex set, and not necessarily a convex set. It is known that, in any finite-dimensional Euclidean space, both the domain and the range of an arbitrary maximal monotone operator are always nearly convex sets (see~\cite[Proposition~6.4.1]{at_2003}). An extended-real-valued function is called a \textit{nearly convex function} if its epigraph is a nearly convex set.
 
 Nearly convex sets have been systematically studied by Bauschke et al.~\cite{bmw2013}, and by Moffat et al.~\cite{mmw2016}. The characterizations of nearly convex sets and nearly convex functions have been investigated by Bo{\c{t}}  et al.~\cite{bgw2007,bkw2008},~ Li and Mastroeni~\cite{LM2019}, and Nam et al.~\cite{nty1}.  
 
 Note that, in~\cite{nty1}, the concept of nearly convex set-valued mappings, the preservation of near convexity of set-valued mappings under various operations, and the theory of generalized differentiation for nearly convex set-valued mappings have been presented.
 
Nearly convex optimization problems, that is, optimization problems with nearly convex objective functions and nearly convex constraint sets, were studied in~\cite{bgw2007,bkw2008,LM2019}.  In~\cite{bgw2007}, the authors showed that the classical Fenchel duality statements remain valid when the functions involved are only nearly convex. In~\cite{bkw2008}, a strong duality result was established for nearly convex optimization problems. As particular cases, strong duality results were obtained for both the Lagrange dual problem and the Fenchel–Lagrange dual problem of optimization problems having nearly convex objective functions and nearly convex inequality cone constraints. Later on, Li and Mastroeni~\cite{LM2019} investigated near equality and near convexity of the solution sets, derived optimality conditions for nearly convex optimization problems, and established global error bounds in this setting. Many new results related to near convexity and its applications have been published in~\cite{nr1,nr2,nr3,nr5,nr4}.
 
 This paper presents new results on nearly convex optimization problems, with a focus on necessary and sufficient optimality conditions. These results are proved by employing several tools related to nearly convex sets and nearly convex functions from~\cite{mmw2016,nty1}, together with a formula for computing the normal cone to a sublevel set of a convex function from~\cite[Proposition~2, p.~206]{IT_1979}. The obtained results are analyzed  by a series of examples and counterexamples.
 
  Our main idea is to associate to each nearly convex optimization problem a uniquely defined convex optimization problem with a lower semicontinuous objective function and a closed constraint set. It turns out that remarkable relationships between the original nearly convex problem and the associated convex problem do exist. Optimality conditions in the form of Fermat's rules can be obtained for both problems. However, as far as we understand, the most effective way to get a Lagrange multiplier rule for a nearly convex optimization problem under a geometrical constraint and functional constraints is to derive it from the Kuhn-Tucker conditions for the associated convex optimization problem.
 
 Since optimization problems with a convex objective function and a nearly convex constraint set have been studied by Ho~\cite{ho1}, Jeyakumar and Mohebi~\cite{jm1}, and Ghafari and Mohebi~\cite{gm2021}, we will make a detailed comparison between two notions of nearly convex sets, namely, the notion of nearly convex sets in the sense of Ho~\cite{ho1} and the notion of nearly convex sets in the sense of Minty mentioned above. This allows us to clearly identify the differences between the optimality conditions presented in this paper and the corresponding results in~\cite{gm2021,ho1,jm1}.
 
The paper is organized as follows. In Sect.~\ref{sect2}, two notions of nearly convex sets are considered and compared. Nearly convex functions, nearly convex optimization problems, the associated convex problems and their properties,  and two kinds of solutions of a nearly convex optimization problem are investigated in Sect.~\ref{nc_problems}. In Sect.~\ref{optimality}, we establish necessary and sufficient optimality conditions for nearly convex optimization problems. Several concluding remarks are given in Sect.~\ref{sect5}.
 
 The topological closure and the interior of a set $D \subset \mathbb{R}^n$ are denoted by $\bar{D}$ and ${\rm int}\, D$, respectively. The affine hull, the relative interior, and the convex hull of $D$ are denoted by  ${\rm aff}, D$, ${\rm ri}\, D$, and ${\rm co}\, D$, respectively. Thus, ${\rm aff}\, D$ is the smallest affine set in $\mathbb{R}^n$ that contains $D$, ${\rm ri}\, D$ is the interior of $D$ with respect to the topology induced on ${\rm aff}\, D$, and ${\rm co}\, D$ is the smallest convex set containing $D$. It is known that ${\rm co}\, D$ is precisely the intersection of all convex sets that contain $D$. Moreover, if $D \subset \mathbb{R}^n$ is a nonempty convex set, then ${\rm ri}\, D \neq \emptyset$ (see~\cite[Theorem~6.2]{r}). The cone generated by a set $D \subset \mathbb{R}^n$, denoted by ${\rm cone}\, D$, is defined by $${\rm cone}\,  D=\big\{\lambda x\mid x\in D,\; \lambda\geq 0\big\}.$$ If $D$ is a convex set, then ${\rm cone}\,  D$ is the intersection of all convex cones containing~$D$ and the origin; hence it coincides with the cone $K_D$ introduced in~\cite[Subsection~3.1.1, p.~162]{IT_1979}  (see~\cite[Proposition~3, p.~163]{IT_1979}). Let $B(x,r)$ and $\bar B(x,r)$ denote the open and closed balls centered at~$x$ with radius~$r$, respectively. The set of positive integers is denoted by $\mathbb{N}$. The notations $\mathbb{R}_+$ and $\R_-$ are used to denote the set of nonnegative real numbers and the set of nonpositive real numbers, respectively.
 
 \section{On Two Notions of Nearly Convex Sets}\label{sect2}

First, we recall two notions of nearly convex sets, where the first notion was proposed by Minty~\cite{Minty1961} in another equivalent form.
 
 \begin{definition}\label{def_Minty} {\rm (See~\cite[Definition~2.6]{bmw2013}) A subset $\Omega$ of $\mathbb{R}^n$ is called {\em nearly convex} if there exists a convex set $C \subset \mathbb{R}^n$ such that	$C \subset \Omega \subset \bar C.$}
 \end{definition}
 
 In other words, a subset of $\R^n$ is nearly convex if it lies between some convex set and the closure of that convex set.
 
 \begin{definition}\label{def_Ho} {\rm (See~\cite[p.~42]{ho1}) A subset $\Omega \subset \mathbb{R}^n$ is called {\em nearly convex in the sense of Ho at a point $x \in \Omega$} if for every point $y \in \Omega$, there exists a sequence $\{t_k\}$ of positive numbers such that $t_k \to 0$ as $k \to \infty$ and $x + t_k (y - x) \in \Omega$ for all $k \in \mathbb{N}$. If~$\Omega$ is nearly convex in the sense of Ho at every point of $\Omega$, then it is called a \textit{nearly convex set in the sense of Ho}.}
 \end{definition}
 
 Since $x + t_k (y - x) = (1 - t_k)x + t_k y$ is a convex combination of $x$ and $y$ for each $t_k \in (0,1)$, the set $\Omega \subset \mathbb{R}^n$ is nearly convex in the sense of Ho at $x \in \Omega$ if and only if, on the open segment $(x,y)$ of the line segment joining $x$ to an arbitrary point $y \in \Omega$, there exists a sequence of vectors belonging to $\Omega$ that converges to $x$.
 
The notion of near convexity at a point in the sense of Ho was used by Jeyakumar and Mohebi~\cite[p.~117]{jm1} and later by Ghafari and Mohebi~\cite[Definition~2.7]{gm2021}.
 
 The following two examples show that nearly convex sets in the sense of Definition~\ref{def_Minty} may fail to be nearly convex in the sense of Definition~\ref{def_Ho}.

 \begin{example}\label{ex1}
 	{\rm  The set  $\Omega_1=\big\{(x_1,x_2)\in\mathbb R^2\mid x_1\geq 0\big\}\setminus \big(\{0\}\times (-1,1)\big)$ is nearly convex, but not nearly convex in the sense of Ho. Indeed, for $x=(0,-1)$ and $y=(0,1)$, we see that on the open segment $(x,y)$ of the line segment $[x,y]$ there exists no sequence of vectors belonging to $\Omega_1$ that converges to $x$.}
 \end{example}
 
 \begin{example}\label{ex2}
 	{\rm  The set  $\Omega_2=([0,1]\times [0,1])\setminus \Big\{x=(1,x_2)\mid \dfrac{1}{2}<x_2<\dfrac{3}{4}\Big\}$ is nearly convex, but not nearly convex in the sense of Ho. Indeed, for $x=(1,\dfrac{1}{2})$ and $y=(1,\dfrac{3}{4})$, we see that on the open segment $(x,y)$ of the line segment $[x,y]$ there exists no sequence of vectors belonging to $\Omega_1$ that converges to $x$.}
 \end{example}
 
 The following examples show that nearly convex sets in the sense of Definition~\ref{def_Ho} may fail to be nearly convex in the sense of Definition~\ref{def_Minty}.
 
 \begin{example}\label{ex3}
 	{\rm  The set of rational numbers $\mathbb{Q}\subset\mathbb{R}$ is nearly convex in the sense of Ho, but not nearly convex. It is also easy to see that the set $\Omega_3=\mathbb Q\times [0,1]\subset \mathbb R^2$ is nearly convex in the sense of Ho, but not nearly convex.}
 \end{example}
 
 \begin{example}\label{ex4} 
 	{\rm  (See~\cite[p.~42]{ho1}) The set  $$\Omega_4=\big\{(x_1,x_2)\in\mathbb R^2\mid x_1^2+x_2^2\leq 25\big\}\setminus\big\{(x_1,x_2)\in\mathbb R^2\mid (x_1-2)^2+x_2^2\leq 1\big\}$$ is nearly convex in the sense of Ho. However, $\Omega_4$ is not nearly convex.}
 \end{example}
 
 \begin{example}\label{ex5} 
 	{\rm  (See~\cite[p.~117]{jm1}) The set  $$\Omega_5=\big\{(x_1,x_2)\in\mathbb R^2\mid x_1^2+x_2^2\leq 16\big\}\cap\big\{(x_1,x_2)\in\mathbb R^2\mid (x_1-2)^2+x_2^2>1\big\}$$ is nearly convex in the sense of Ho. Nevertheless, $\Omega_5$ is not nearly convex.}
 \end{example}
 
 The above examples show that \textit{the notions of near convexity in the sense of Definition~\ref{def_Minty} and Definition~\ref{def_Ho} are completely different}. Therefore, the optimality conditions in~\cite{gm2021,ho1,jm1} are entirely different from the optimality conditions that will be established herein.

 \textit{Throughout the remainder of the present paper, near convexity of sets will always be understood in the sense of Definition~\ref{def_Minty}.}

 By~\cite[Lemma~2.7]{bmw2013} (see also~\cite[Proposition~2.1]{nty1}), if $\Omega\subset\R^n$ is nearly convex, then ${\rm ri}\,\Omega$ and $\overline{\Omega}$ are convex sets, and ${\rm ri}\,\Omega= {\rm ri}\,\overline{\Omega}$. Moreover, near convexity is preserved under linear mappings (and hence under affine mappings as well). In addition, the intersection of finitely many nearly convex sets is nearly convex, provided that a certain condition on the relationship among the relative interiors of these sets is satisfied.

 \begin{proposition}\label{T1} {\rm (See~\cite[Theorem~4.2 and Corollary~4.8]{mmw2016})} Let $\Omega$, $\Omega_1,\dots, \Omega_m$ be nearly convex sets in $\R^n$, and $A\colon \R^n\to \R^p$ be a linear operator. Then
 	\begin{enumerate}
 		\item[{\rm (a)}] $A(\Omega)$ is a nearly convex set in $\R^p$, $\ri A(\Omega)=A(\ri\Omega)$, and $\overline{A(\Omega)}=\overline{A(\overline{\Omega})}$.
 		\item[{\rm (b)}] If $\bigcap\limits_{i=1}^m\ri\Omega_i\neq\emptyset$, then $\bigcap\limits_{i=1}^m\Omega_i$ is nearly convex,
 		\begin{equation*}
 			\ri\left(\bigcap_{i=1}^m\Omega_i\right)=\bigcap_{i=1}^m\ri\Omega_i,
 		\end{equation*}
 		and
 			\begin{equation*}
 			\overline{\bigcap_{i=1}^m\Omega_i}=\bigcap_{i=1}^m\overline{\Omega_i}.
 		\end{equation*}
 	\end{enumerate}
 \end{proposition}

 \begin{proposition}\label{Connection} {\rm (See~\cite[Proposition~3.1]{nty1})}  Let $\Omega$ be a nearly convex set in $\R^n$. If $a\in \ri\Omega$ and $b\in \overline{\Omega}$, then $[a, b)\subset \ri\Omega.$
 \end{proposition}
 
 \begin{proposition}\label{ri_prod}{\rm (See~\cite[Proposition~4.1]{mmw2016}) }
Let $\Omega_1,\dots, \Omega_m$ be nearly convex sets in $\R^n$. Then,
$\mbox{\rm ri}(\Omega_1\times\dots\times \Omega_m)=\ri\Omega_1\times\ri \Omega_m.$
 	\end{proposition}

 \begin{definition}\label{normal_cone} {\rm Given a nearly convex set $\Omega$ in $\R^n$ and $\ox\in \Omega$, we define the {\em normal cone} to $\Omega$ at $\ox$ by setting
 		\begin{equation*}
 			N(\ox; \Omega)=\big\{x^*\in \R^n\mid  \langle  x^*, x-\ox\rangle \leq 0\ \, \mbox{\rm for all }\, x\in \Omega\big\}.
 		\end{equation*}  For any $\ox\notin \Omega$, we put $N(\ox; \Omega)=\emptyset$.
 	}
 \end{definition}
 
 \begin{proposition}\label{NCI} {\rm (See~\cite[Theorem~5.1]{nty1})} Let $\Omega_1$ and $\Omega_2$ be nearly convex sets  in~$\mathbb R^n$ such that
 	$\ri\Omega_1\cap \ri\Omega_2\neq \emptyset$.
 	Then $\Omega_1\cap \Omega_2$ is nearly convex and
 	\begin{equation*}
 		N(\ox; \Omega_1\cap \Omega_2)=N(\ox; \Omega_1)+N(\ox; \Omega_2)\ \, \mbox{for all }\, \ox\in \Omega_1\cap \Omega_2.
 	\end{equation*}
 \end{proposition}
 
 \section{Nearly Convex Optimization Problems}\label{nc_problems}

First, let us recall the concept of nearly convex function and some known results.
 
\subsection{Nearly convex functions} 

Given an extended-real-valued function $f\colon \R^n\to \oR=[-\infty, \infty]$, the {\em effective domain} and the  {\em epigraph} of $f$ are given  respectively by $\dom f=\big\{x\in \R^n\; \big|\; f(x)<\infty\big\}$ and
$$\epi f=\big\{(x, \lambda)\in \R^n\times \R\; \big|\; f(x)\leq \lambda\big\}.$$
We say that $f$ is {\em proper} if $\dom f\neq\emptyset$ and  $f(x)>-\infty$ for all $x\in \R^n$. 
 
 \begin{definition}\label{nc_function} {\rm (See, e.g.,~\cite[p.~608]{nty1})} {\rm  A function $f\colon \R^n\to \oR$ said to be {\em nearly convex} if $\epi f$ is a nearly convex set in  $\R^n\times\mathbb R$.
 	}
 \end{definition}

The following proposition concerns the near convexity property of the effective domain of a nearly convex function.
 
 \begin{proposition}\label{nc dom} {\rm (See~\cite[Proposition~3.4]{nty1})} If  $f\colon\R^n\to \oR$ is a nearly convex function, then $\dom f$ is a nearly convex set.
 \end{proposition}

 \begin{proposition}\label{nc sum f}  {\rm (See~\cite[Corollary~4.3]{nty1})} Suppose that $f_1, f_2\colon \R^n\to \oR$ are proper nearly convex functions. If
  		$\mbox{\rm ri}(\dom f_1)\cap \mbox{\rm ri}(\dom f_2)\neq \emptyset$,
 	then $f_1+f_2$ is nearly convex.
 \end{proposition}
 
The next proposition gives a formula for the relative interior of the epigraph of a nearly convex function, which extends and deepens the result in~\cite[Proposition~4.4]{LM2019}.
 
 \begin{proposition}\label{riepi} {\rm (See~\cite[Proposition~3.7]{nty1})} Let $f\colon \R^n\to \oR$ be an arbitrary function. Then,
 	\begin{equation}\label{epi rep1}
 		\mbox{\rm ri}(\epi f)\subset\big\{(x, \lambda)\in \R^n\times \R\mid x\in \mbox{\rm ri}(\dom f),\ \lambda>f(x)\big\}.
 	\end{equation}
 	In addition, if $f$ is nearly convex, then the reverse inclusion~\eqref{epi rep1} holds, i.e.,
 	\begin{equation}\label{epi rep1a}
 		\mbox{\rm ri}(\epi f)=\big\{(x, \lambda)\in \R^n\times \R\mid x\in \mbox{\rm ri}(\dom f),\ \lambda>f(x)\big\}.
 	\end{equation}
 \end{proposition}

\subsection{Subdifferential}  

 \begin{definition}\label{subdifferential} {\rm  Let $f\colon \R^n\to \oR$ be an extended-real-valued function. The \textit{subdifferential} in the sense of convex analysis of $f$ at $\ox\in \R^n$ with $f(\ox)\in\R$ is defined by
 		\begin{equation*}
 			\partial f(\ox)=\big\{x^*\in \R^n\mid  \langle  x^*, x-\ox\rangle \leq f(x)-f(\ox)\ \, \mbox{\rm for all }\, x\in \R^n \big\}.
 		\end{equation*}
 	}
 \end{definition}
 
The subdifferential sum rule for proper, nearly convex functions is stated as follows.
 
 \begin{proposition}\label{sum_rule} {\rm (See~\cite[Corollary~4.3]{nty1})} Let $f_i\colon \R^n\to \oR$, $i=1, \ldots, m$ be proper nearly convex functions. If
  	$\bigcap\limits_{i=1}^m \mbox{\rm ri}(\dom f_i)\neq\emptyset$,
 	then the sum function $f_1+\cdots+f_m$ is nearly convex and the equality
 	\begin{equation*}
 		\partial (f_1+\cdots +f_m)(\ox)=\partial f(\ox)+\cdots +\partial f_m(\ox)
 	\end{equation*} holds for every $\ox\in \bigcap\limits_{i=1}^m \dom f_i$.
 \end{proposition}
 
 By definition, a function $f\colon \R^n\to \oR$ is continuous at $\ox\in \R^n$ if $\ox\in \mbox{\rm int}(\dom f)$ and for every $\varepsilon>0$ there is  $\delta>0$ such that $B(\ox; \delta)\subset \dom f$ and $|f(x)-f(\ox)|<\varepsilon$ for all $x\in B(\ox; \delta)$.

The \textit{maximum function} of a finite family of functions $f_i\colon \R^n \to \oR$, $i=1, \ldots, m$ is defined by
 \begin{equation}\label{MF}
 	f(x)=\max\big\{f_i(x)\mid  i=1, \ldots, m\big\},\ \; x\in \R^n.
 \end{equation}

 \begin{proposition}\label{NCMf} {\rm (See~\cite[Corollary~4.10]{nty1})} Suppose that $f_i\colon \R^n\to \oR$ for $i=1, \ldots, m$ are nearly convex functions. If
  		$\bigcap\limits_{i=1}^m \mbox{\rm ri}(\dom f_i)\neq\emptyset$,
  	then the function $f$ defined in~\eqref{MF} is nearly convex.
 \end{proposition}
 
The subdifferential of a maximum function can be computed by using the next proposition.
 
 \begin{proposition}\label{diff_max} {\rm (See~\cite[Corollary~5.10]{nty1})} Let $f_i\colon \R^n\to \oR$  for $i=1, \ldots, m$  be proper nearly convex functions. Suppose that all the functions $f_i$ are continuous at $\ox\in \R^n$. Then, the function $f$ defined by~\eqref{MF} is also nearly convex, and one has
 	\begin{equation*}
 		\partial f(\ox)=\co \big[\bigcup_{i\in I(\ox)}\partial f_i(\ox)\big],
 	\end{equation*}
 	where $I(\ox)=\big\{i=1, \ldots, m\mid f_i(\ox)=f(\ox)\big\}$.
 \end{proposition}
 
 \subsection{Nearly convex optimization problems and some auxiliary facts}
  
Let $f\colon \R^n\to \oR$ be a proper function and let $D\subset\R^n$ be a nonempty set. Consider the optimization problem
 \begin{equation}\label{optim-1}\min\{f(x)\mid x\in D\}.
 \end{equation}

 \begin{definition}\label{nc_prob1} {\rm   If $f$ is a nearly convex function and $D$ is a nearly convex set, then we call~\eqref{optim-1} a \textit{nearly convex optimization problem}  and denote its solution set by~$\mathcal{S}$.}
 \end{definition}

 It is well known that the solution set of a convex optimization problem, that is, the problem of minimizing a convex function over a convex set, is a convex set.  In view of these facts, the following question is worth considering.

 \textbf{Question 1.} \textit{Is the solution set of a nearly convex optimization problem a nearly convex set? If this is not true in general, under what conditions is the solution set of a nearly convex optimization problem nearly convex?}

The next examples answer Question~1 in the negative. In the first example, the objective function is linear and the constraint set is bounded.  In the second one, the objective function is convex linear-quadratic and the constraint set is unbounded.

\begin{example}\label{eg:3.1}
	{\rm Consider the nearly convex optimization problem in the form~\eqref{optim-1} where $f:\R^2\to \R$ is defined by $f(x_1,x_2)=x_1$ for $(x_1,x_2)\in\R^2$ and $$D=\big([0,1]\times [0,1]\big)\setminus\left(\{0\}\times \left[\frac{1}{4},\frac{3}{4}\right]\right).$$ Here, the optimal value of~\eqref{optim-1} is $0$. The value is attained at $(x_1,x_2)\in D$ when $x_1=0$. Thus, the solution set is
		$\mathcal S=\{0\}\times \left(\big[0,\dfrac{1}{4}\big)\cup \big(\frac{3}{4},1\big]\right).$
		Since $${\rm ri}\,\mathcal S=\{0\}\times \left(\big(0,\dfrac{1}{4}\big)\cup \big(\frac{3}{4},1\big)\right)$$ is nonconvex, the set $\mathcal S$ is not nearly convex. 			
	}
\end{example}
 	
 \begin{example}\label{eg:3.2}
 		{\rm Consider the problem of the type~\eqref{optim-1} with $f:\R^2\to \R$ being defined by $f(x_1,x_2)=x_1^2+x_1$ for $(x_1,x_2)\in\R^2$ and $$D=\big(\R_+\times \R\big)\setminus\big(\{0\}\times (-1,1)\big).$$ Clearly, this problem is nearly convex. Note that the optimal value is $0$ and the solution set is
 			$\mathcal S=\{0\}\times\big( (-\infty,-1]\cup [1,+\infty)\big).$
 			Note that $\mathcal S$ is not nearly convex, because	${\rm ri}\,\mathcal S=\{0\}\times \big((-\infty,-1)\cup (1,+\infty)\big)$	is a nonconvex set.
 		}
 \end{example}

The following question remains open.
  
\textbf{Question 2.} \textit{Under what conditions is the solution set of a nearly convex optimization problem nearly convex?}

\smallskip
Let $f\colon \R^n\to \oR$ be a function and $\bar{f}:\R^n\to \oR$ be defined by setting
\begin{equation}\label{bar_f}
	\bar{f}(x)=\inf\big\{\alpha\in \R \mid (x, \alpha)\in \overline{\epi f}\,\big\}
\end{equation}
for every $x\in \R^n$. One has $\bar f(x)\leq f(x)$ for every $x\in \R^n$ and $\dom f\subset \dom \bar f$. Indeed, take any $x\in\R^n$. If $f(x)=+\infty$, then the inequality is obvious. If $f(x)=-\infty$, for any $\alpha\in\R$, it holds that $(x,\alpha)\in \epi f$; hence~\eqref{bar_f} implies that $\bar f(x)=-\infty$. If $f(x)\in\R$, then $(x,f(x))\in \epi f$. So, by~\eqref{bar_f},~$\bar f(x)\leq f(x)$. We have thus shown that not only $\bar f(x)\leq f(x)$ for every $x\in \R^n$, but also $\dom f\subset \dom \bar f$. Note that the strict inequality $\bar f(x)<f(x)$ may hold for some $x\in\dom f$.
 
\begin{lemma}\label{lem:epi_closure}
For any function $f\colon \R^n\to \oR$, if $(x,\alpha)\in \overline{\epi f}$ and $\beta>\alpha$, then one has $(x,\beta)\in \overline{\epi f}$.
\end{lemma}
\begin{proof}
Since $(x,\alpha)\in \overline{\epi f}$, there exists a sequence $(x_k,\alpha_k)\in\epi f$ such that
$$(x_k,\alpha_k)\to (x,\alpha)\quad\text{as } k\to\infty.$$
As $\alpha<\beta$, we can find an index $k_0\in\mathbb N$ such that $\alpha_k<\beta$ for all $k\geq k_0$. Then, for every $k\geq k_0$, the inclusion $(x_k,\alpha_k)\in\epi f$ implies that  $(x_k,\beta)\in \epi f$. Letting $k\to\infty$, from the last inclusion we get $(x,\beta)\in \overline{\epi f}$.
$\hfill\Box$ 
\end{proof}

According to Rockafellar~\cite[p.~52]{r}, for any function $f\colon \R^n\to \oR$, there exists a greatest lower semicontinuous function (not necessarily finite) majorized by $f$, that is the function whose epigraph is the topological closure of $\epi f$. This function is called the \textit{lower semicontinuous hull} of $f$.

For the sake of clarity, the fact that the function $\bar f$ defined by~\eqref{bar_f} is the lower semicontinuous hull of $f$ is proved in detail in the following lemma.

 \begin{lemma}\label{lem:bar_f}
 	For any function $f\colon \R^n\to \oR$, one has 
 	\begin{equation}\label{epi_bar_f}
 		\epi\bar f= \overline{\epi f},
 	\end{equation}
 	where the function $\bar f$ is defined by~\eqref{bar_f}. Thus, $\bar f$ is lower semicontinuous hull of $f$. If, in addition, $f$  is proper and nearly convex, then $\bar{f}$ is proper and convex.
 \end{lemma}
 
 \begin{proof} To obtain the equality~\eqref{epi_bar_f}, we can argue as follows. First, take any $(x,\gamma)\in \overline{\epi f}$ and note that $$\gamma\in \left\{\alpha\in \R \mid (x, \alpha)\in \overline{\epi f}\,\right\}.$$ Then, by~\eqref{bar_f} we have $\bar f(x)\le\gamma$. This means that $(x,\gamma)\in\epi\bar f$. So, we can infer that $\overline{\epi f}\subset\epi\bar f$. Now, take any $(x,\gamma)\in\epi\bar f$ and observe that $ \gamma\geq \bar f(x).$ By~\eqref{bar_f}, we can find a sequence $\{\alpha_k\}\subset \R$ with $\alpha_k\downarrow \bar f(x)$ and $(x,\alpha_k)\in\overline{\epi f}$.  If $\gamma>\bar f(x)$, then there exists $k_0\in\N$ such that $\gamma > \alpha_k$ for all $k\geq k_0$. As $(x,\alpha_k)\in\overline{\epi f}$ for all $k\geq k_0$, by Lemma~\ref{lem:epi_closure} we have  $(x,\gamma)\in \overline{\epi f}$. If $\gamma = \bar f(x)$, then $\bar f(x)\in \R$. As $\overline{\epi f}$ is closed, passing the inclusion $(x,\alpha_k)\in \overline{\epi f}$ to the limit as $k\to\infty$ yields $(x,\bar f(x))\in \overline{\epi f}$. Therefore, $(x,\gamma)\in \overline{\epi f}$.  We have proved that $\epi\bar f\subset \overline{\epi f}$. So, the equality~\eqref{epi_bar_f} holds true. 
 	
 	Now, suppose that $f$  is proper and nearly convex. Since $\dom f\subset \dom \bar f$ and $\dom f\neq\emptyset$, we have $\dom \bar f\neq\emptyset$. So, if the function $\bar{f}$ is improper, then there exists $y\in\dom \bar f$ such that $\bar f(y)=-\infty$. Clearly, the equality $$\inf\left\{\alpha\in \R \mid (y, \alpha)\in \overline{\epi f}\,\right\}=-\infty$$
 	implies the existence of a sequence $\{\alpha_k\}\subset\R$ with $\lim\limits_{k\to\infty}\alpha_k=-\infty$ and $(y,\alpha_k)\in\overline{\epi f}$ for all $k\in\mathbb N$. Because $f$ is a nearly convex function, $\epi f$ is a nearly convex set by Definition~\ref{nc_function}. Therefore, according to~\cite[Proposition~2.1]{nty1}, $\overline{\epi f}$ is a closed convex set and ${\rm ri}(\epi f)={\rm ri}(\overline{\epi f})\neq\emptyset$. 
 	Pick any $(\bar x,\bar\lambda)\in{\rm ri}(\epi f)$  and fix a number $t\in (0,1)$. For each $k\in\mathbb N$, 
 	by the inclusion $(y,\alpha_k)\in\overline{\epi f}$ and Proposition~\ref{Connection} we get
 	$$
 	(1-t)(\bar x,\bar\lambda)+t(y,\alpha_k)\in{\rm ri}(\epi f).
 	$$ 	
    Hence, thanks to the equality~\eqref{epi rep1a} in Proposition~\ref{riepi}, we have \begin{equation}\label{eqn_n1}
    	(1-t)\bar x+ty\in{\rm ri}(\dom f)
    \end{equation} and
 	$$
 	(1-t)\bar\lambda+t\alpha_k
 	>
 	f\big((1-t)\bar x+ty\big).
  	$$ Letting $k\to\infty$, from the last inequality we can deduce that 
 	$f\big((1-t)\bar x+ty\big)=-\infty.$ This contradicts the inclusion~\eqref{eqn_n1}. We have thus proved that 
 	$\bar f$ is a proper function.

Since $f$ is nearly convex, the set $\overline{\epi f}$ is closed and convex. So, by~\eqref{epi_bar_f} we can assert that the function~$\bar f$ is lower semicontinuous and convex.
$\hfill\Box$
\end{proof}

The next lemma gives us an alternative formula to effectively compute the value of the function defined by~\eqref{bar_f} at any point in $\mathbb R^n$.

\begin{lemma}\label{lem:bar_f_lsc}
    For any function $f\colon \R^n\to \oR$, one has 
	\begin{equation}\label{bar_f(y)}
	\bar f(y)=\liminf_{x\to y} f(x)
	\end{equation}
	for any point $y\in\mathbb R^n$, where the function $\bar f$ is defined by~\eqref{bar_f}. 
\end{lemma}

\begin{proof} Given any $y\in\mathbb R^n$, we put $\gamma=\liminf\limits_{x\to y} f(x)$.  Since
	\begin{equation*}
		\bar{f}(y)=\inf\left\{\alpha\in \R \mid (y, \alpha)\in \overline{\epi f}\,\right\},
	\end{equation*}  there exists a sequence $\{\alpha_k\}\subset\R$ such that  $\lim\limits_{k\to\infty}\alpha_k=\bar{f}(y)$  and $(y, \alpha_k)\in \overline{\epi f}$ for all $k\in\N$. For each $k\in\mathbb N$, select a point $(y_k,\beta_k)\in\epi f$ such that  $\|y_k-y\|\leq\dfrac{1}{k}$ and $|\beta_k-\alpha_k|\leq \dfrac{1}{k}$.
	Since $\beta_k\geq f(y_k)$ for all $k$, we have 
		\begin{equation*}
			\lim\limits_{k\to\infty}\alpha_k=\lim\limits_{k\to\infty}\beta_k=\liminf\limits_{k\to\infty}\beta_k\geq \liminf\limits_{k\to\infty}f(y_k)\geq \liminf\limits_{x\to y} f(x).
	\end{equation*} It follows that $\bar f(y)\geq\gamma$. To obtain the reverse inequality, we first observe that if $\gamma=+\infty$, then  $\bar f(y)\leq\gamma$. Now, suppose that $\gamma<+\infty$. Let $\{x_k\}$ be a sequence in~$\mathbb R^n$ tending to $y$ such that $\lim\limits_{k\to\infty}f(x_k)=\gamma$. Since $\gamma<+\infty$, either there exists an index $k_0$ such that $f(x_k)\in\mathbb R$ for all $k\geq k_0$ or there is a subsequence $\{x_{k_j}\}$ of $\{x_k\}$ with $f(x_{k_j})=-\infty$ for all $j\in\mathbb N$. 
	
	In the first case, for each $k\geq k_0$, as $(x_k,f(x_k))\in\epi f$, by~\eqref{bar_f} we see that $\bar f(x_k)\leq f(x_k)$. This implies that 
	$$\liminf\limits_{k\to\infty}\bar f(x_k)\leq \liminf\limits_{k\to\infty}f(x_k)=\gamma.$$ Meanwhile, by the lower semicontinuity of $\bar f$, $\liminf\limits_{k\to\infty}\bar f(x_k)\geq \bar f(y)$. So, we have $\bar f(y)\leq\gamma$. 
	
	In the second case, for any $\beta\in\mathbb R$, since $(x_{k_j},\beta) \in\epi f$ for all $j\in\mathbb N$, from the equality $\lim\limits_{j\to\infty} (x_{k_j},\beta)=(y,\beta)$ we get $(y,\beta)\in \overline{\epi f}$. As this inclusion is fulfilled for any $\beta\in\mathbb R$, from~\eqref{bar_f} it follows that $\bar f(y)=-\infty$. Therefore, $\bar f(y)\leq\gamma$.
	
	Summing up all the above, we can infer that the equality~\eqref{bar_f(y)} is valid.
	$\hfill\Box$
\end{proof}

Since the notion of almost convex function in~\cite[Definition~3(i)]{bgw2007} is equivalent to the notion of nearly convex function in Definition~\ref{nc_function}, the forthcoming lemma restates one assertion of an important result in~\cite[Theorem~1]{bgw2007}. For the completeness of the present paper, we give a new and detailed proof for this lemma.

\begin{lemma}\label{eq_in_ridom} Suppose that $f\colon \R^n\to \oR$ is a proper nearly convex function and $\bar f$ is defined by~\eqref{bar_f}. Then,  $\bar f$ is a convex function, ${\rm ri} (\dom \bar f)={\rm ri}(\dom f)$ and $\bar f(x)=f(x)$ for every $x\in {\rm ri}(\dom f)$. 
\end{lemma}
\begin{proof}  By Lemma~\ref{lem:bar_f}, $\epi\bar f= \overline{\epi f}$. So, as $\epi f$ is a nearly convex set,  $\epi\bar f$ is a convex function (hence $\bar f$ is a convex function) and we have
	$${\rm ri}(\epi f)= {\rm ri}(\overline{\epi f})= {\rm ri}(\epi \bar f).$$
Therefore, applying the second assertion of Proposition~\ref{riepi} to both functions~$f$ and~$\bar f$ we have 
\begin{eqnarray*}\label{ri_epi1}
	{\rm ri}(\epi f) = \left\{(x,\lambda)\in \R^n\times \R \mid x\in {\rm ri}(\dom f),\ \lambda>f(x)\right\}
	\end{eqnarray*} and
	\begin{eqnarray*}\label{ri_epi}
	{\rm ri}(\epi \bar f) =  \left\{(x,\eta)\in \R^n\times \R \mid x\in {\rm ri}(\dom \bar f),\ \eta>\bar f(x)\right\}.
	\end{eqnarray*} 
	Combining this with the equality ${\rm ri}(\epi f)={\rm ri}(\epi \bar f)$ yields 
	 \begin{eqnarray}\label{ri_epis}\begin{array}{rl}
	 	& \big\{(x,\lambda)\in \R^n\times \R \mid x\in {\rm ri}(\dom f),\ \lambda>f(x)\big\}\\
	 	& =\big\{(x,\eta)\in \R^n\times \R \mid x\in {\rm ri}(\dom \bar f),\ \eta>\bar f(x)\big\}.
	 	\end{array}
	 \end{eqnarray} It follows that
	\begin{equation}\label{eq_ri_doms}
		{\rm ri}(\dom f)={\rm ri} (\dom \bar f).
	\end{equation}
 Moreover, if there exists some point $x\in {\rm ri}(\dom f)$ with $\bar f(x)\neq f(x)$, then one gets $f(x)>\bar f(x)$ and it is clear that~\eqref{ri_epis} cannot hold. Consequently, we must have $\bar f(x)=f(x)$ for every $x\in {\rm ri}(\dom f)$. 
$\hfill\Box$
\end{proof}

\subsection{The associated convex problem and its properties}

\begin{definition}
	{\rm If~\eqref{optim-1} is a nearly convex optimization problem, then the optimization problem 
	 \begin{equation}\label{optim-2}\min\{\bar{f}(x)\mid x\in \bar{D}\}
 \end{equation}
	where $\bar f$ is defined by~\eqref{bar_f}, is said to be its \textit{associated convex problem}. The solution set of~\eqref{optim-2} is denoted by~$\mathcal{S}_1$.
	}
\end{definition}

 \begin{theorem}\label{opt_values}
 If~\eqref{optim-1} is a nearly convex optimization problem, then its optimal value is equal to that of the associated convex problem~\eqref{optim-2}, provided that the regularity condition
 \begin{equation}\label{reg1}
 	(\ri D)\cap \big({\rm ri}(\dom f)\big)\neq \emptyset
 \end{equation} is satisfied.
 \end{theorem}
 \begin{proof}
 Let
 	\begin{equation}\label{eq:op_val}
 	v:=\inf\{f(x)\mid x\in D\}\quad \text{and}\quad \bar{v}:=\inf\{\bar{f}(x)\mid x\in \bar{D}\}.
 	\end{equation}
 	Since $D\subset \bar{D}$ and $\bar{f}(x)\leq f(x)$ for all $x\in\R^n$, by~\eqref{eq:op_val} we have $\bar{v}\leq v.$ It remains to prove that $\bar{v}\geq v.$ 
 	
 	As $D$ is a nonempty nearly convex set, $\ri D$ and $\bar{D}$ are nonempty convex sets, and $\ri D=\ri\bar{D}$ (see~\cite[Lemma~2.7]{bmw2013},~\cite[Proposition~2.1]{nty1}, and~\cite[Theorem~6.2]{r}). 
 	
 	\smallskip
 	{\sc Claim 1}. \textit{One has
 	\begin{equation}\label{eq_infs}
 		\inf\{\bar f(x)\mid x\in \ri D\}=\inf\{\bar{f}(x)\mid x\in \bar{D}\}.
 	\end{equation}}
 	
 	\smallskip
 	Indeed, since $\ri D \subset \bar{D}$, it is immediate that
 	\begin{equation}\label{geq_alpha}
 		\inf\{\bar f(x)\mid x\in \ri D\} \ge\bar v.
 	\end{equation}
  Let $\{x_k\}\subset \bar{D}$ be such a sequence that $\lim\limits_{k\to\infty} \bar f(x_k)=\bar v.$ Thanks to the condition~\eqref{reg1}, there is a point $y\in (\ri D)\cap (\dom f)$. For every $t\in[0,1)$, by Proposition~\ref{Connection} one gets
 	$$
 	(1-t)y+t x_k \in \ri D \qquad \forall k\in\mathbb N.
 	$$
  On one hand, by the convexity of $\bar f$,
 	$$
 	\bar f\big((1-t)y+t x_k\big)
 	\le (1-t)\bar f(y)+t\bar f(x_k)
 	\qquad \forall k\in\mathbb N.
 	$$
 Taking $\limsup$ as $k\to\infty$ on both sides of the last inequality yields
 	$$
 	\limsup_{k\to\infty}\bar f\big((1-t)y+t x_k\big)
 	\le (1-t)\bar f(y)+t\bar v.
 	$$
 	On the other hand, since $(1-t)y+t x_k\in \ri D$ for all $k$, it follows that
 	$$
 	\inf\{\bar f(x)\mid x\in \ri D\}
 	\le \bar f\big((1-t)y+t x_k\big)
 	\qquad \forall k\in\mathbb N.
 	$$
 	Consequently,
 	$$
 	\inf\{\bar f(x)\mid x\in \ri D\}
 	\le \limsup_{k\to\infty}\bar f\big((1-t)y+t x_k\big),
 	$$
 	and hence, for every $t\in[0,1)$,
 	\begin{equation}\label{t_bar_v}
 		\inf\{\bar f(x)\mid x\in \ri D\}
 		\le (1-t)\bar f(y)+t\bar v.
 	\end{equation}
 	As $y\in \dom f$ and the function $\bar f$ is proper by Lemma~\ref{lem:bar_f}, it holds that $$-\infty <\bar f (y)\leq f(y)<+\infty.$$ So, letting $t\uparrow 1$, from~\eqref{t_bar_v} we obtain $\inf\{\bar f(x)\mid x\in \ri D\} \le \bar v.
 	$ Combining this with the inequality~\eqref{geq_alpha},
 	we obtain~\eqref{eq_infs}.
 	
 	\smallskip
 	By~\eqref{eq:op_val},~\eqref{eq_infs}, and the properness of $\bar f$, one has 
 	\begin{equation}\label{eqs_1}\begin{array}{rcl}
 			\bar{v} =  \inf\{\bar{f}(x)\mid x\in \bar{D}\}
 			& = & \inf\{\bar f(x)\mid x\in \ri D\}\\
 			& = & \inf\{\bar f(x)\mid x\in (\ri D)\cap (\dom\bar f)\}.
 		\end{array}
  	\end{equation} Now, from~\eqref{eq_ri_doms} and~\eqref{reg1} we can deduce that $$(\ri D)\cap ({\rm ri}(\dom\bar f))=(\ri D)\cap ({\rm ri}(\dom f))\neq\emptyset.$$ Hence, we can apply Proposition~\ref{T1}(b) to the sets $\ri D$ and $\dom\bar f$ to have 
  	\begin{equation}\label{ri_cap}
  		{\rm ri}\left((\ri D)\cap (\dom\bar f)\right)=(\ri D)\cap ({\rm ri}(\dom\bar f))=(\ri D)\cap ({\rm ri}(\dom f)).
  	\end{equation}
 Put $D_0=(\ri D)\cap (\dom\bar f)$ and note that $D_0$ is a convex set, which may be non-closed. By~\eqref{reg1} and~\eqref{ri_cap} we have 
 \begin{equation}\label{ri_cap_D0}
 \big(\ri D_0\big)\cap \big({\rm ri}(\dom f))\neq\emptyset.
 \end{equation} An analysis of the proof of Claim~1 shows that, thanks to~\eqref{ri_cap_D0}, 
 $$\inf\{\bar f(x)\mid x\in \ri D_0\}= \inf\{\bar f(x)\mid x\in D_0\}.$$ From this,~\eqref{eqs_1},~\eqref{ri_cap}, and Lemma~\ref{eq_in_ridom}, we get
 \begin{equation*}\label{eqs_2}\begin{array}{rcl}
 		\bar{v} & = & \inf\{\bar f(x)\mid x\in (\ri D)\cap (\dom\bar f)\}\\
 		& = & \inf\{\bar f(x)\mid x\in {\rm ri}\left((\ri D)\cap (\dom\bar f)\right)\}\\ 		
 			& = & \inf\{\bar f(x)\mid x\in (\ri D)\cap ({\rm ri}(\dom f))\}\\
 			& = & \inf\{f(x)\mid x\in (\ri D)\cap ({\rm ri}(\dom f))\}\\
 			& \geq & \inf\{f(x)\mid x\in D\cap (\dom f)\}\\
 			& = & \inf\{f(x)\mid x\in D\}.
 		\end{array}
 \end{equation*} So, $\bar{v}\geq v$. 
 
 The proof is complete.
 	 $\hfill\Box$
 \end{proof}

The following example shows that the conclusion of Theorem~\ref{opt_values} can be invalid if the regularity condition~\eqref{reg1} is violated.

\begin{example}\label{eg:opt_vals}
Consider  the set $D=\Big([0,1]\times[0,1]\Big)\setminus\left(\{0\}\times\Big(\dfrac{1}{4},\dfrac{2}{3}\Big)\right)$ and the function $f:\R^2\to \overline{\R}$ defined by
	$$f(x)=\begin{cases} \left|x_2-\dfrac{1}{2}\right| \quad&  \mbox{\rm if }\; x=(x_1,x_2)\in (-\R_+)\times \R,\\
		+\infty & \mbox{\rm otherwise;}\end{cases}$$  see Fig.~\ref{fig:eg:opt_vals1} and Fig.~\ref{fig:eg:opt_vals}.
Clearly, $f$ is a proper lower semicontinuous convex function and $D$ is a nearly convex set. Note that $\dom f= (-\R_+)\times \R$, the function $\bar f$ defined by~\eqref{bar_f} coincides with $f$, $\bar{D}= [0,1]\times[0,1]$, and $\ri D = {\rm int}\, D = (0,1)\times (0,1)$. Since $$\ri D\cap{\rm ri}(\dom f)=\big[(0,1)\times (0,1)\big]\cap \big[(-\infty,0)\times \R\big]=\emptyset,$$ condition~\eqref{reg1} fails to hold. It is easily verified that 
	$$\min\big\{f(x)\mid x\in D\big\}= \dfrac{1}{6},$$
	and the minimum is attained at $x=(0,\frac{2}{3})$. It is also clear that
	$$\min\big\{\bar{f}(x)\mid x\in \bar{D}\big\}=0,$$
	and the minimum is attained at $x=(0,\frac{1}{2})$. Thus, the optimal value of~\eqref{optim-1} is strictly larger than that of~\eqref{optim-2}.
\end{example}

\begin{figure}[h]
	\centering
	\includegraphics[width=0.95\textwidth]{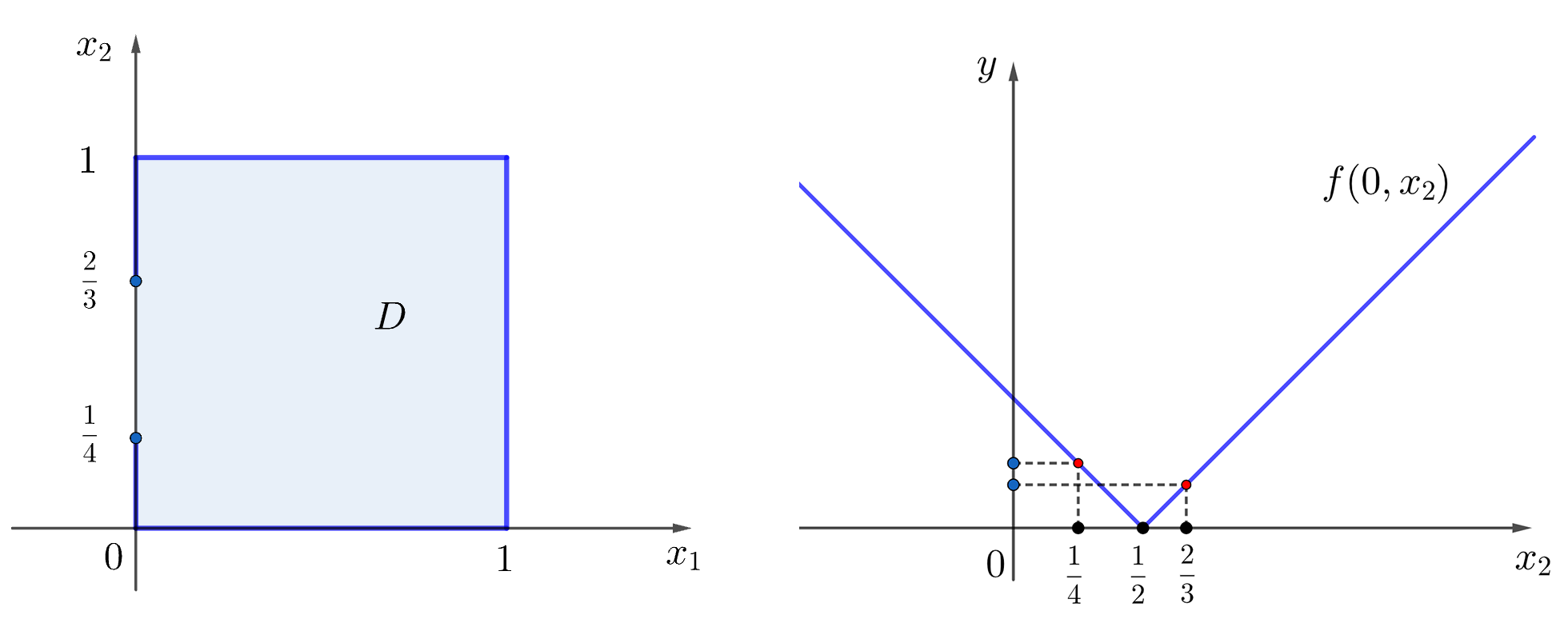}
	\caption{The set $D$ and the graph of the function $f(0,\cdot)$ in Example~\ref{eg:opt_vals}}
	\label{fig:eg:opt_vals1}
\end{figure}

\begin{figure}[h]
	\centering
	\includegraphics[width=0.7\textwidth]{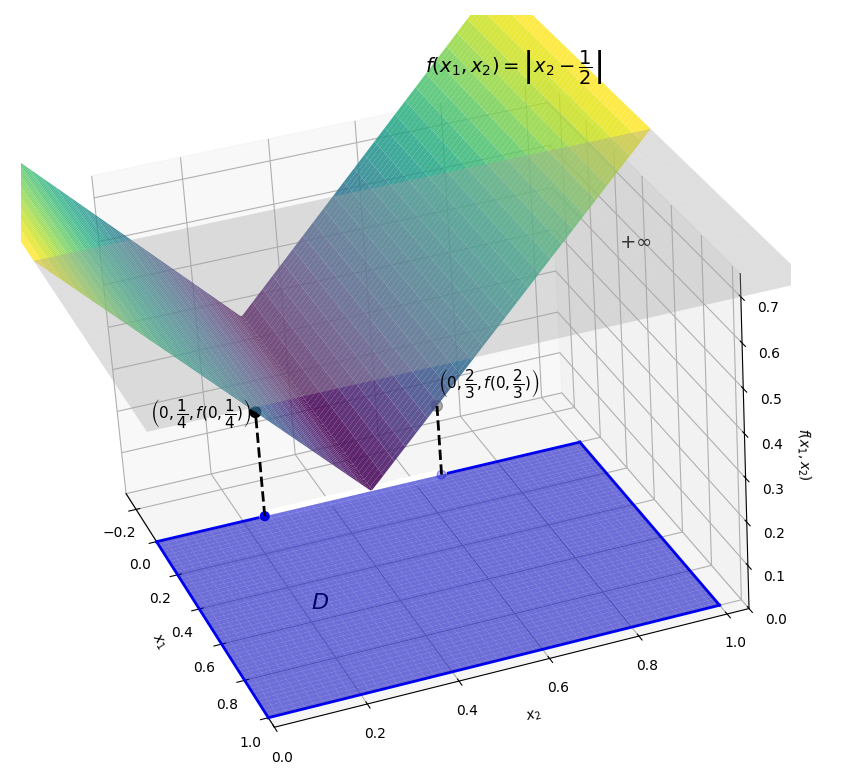}
	\caption{The graph of function $f$ and the set $D$ in Example~\ref{eg:opt_vals}}
	\label{fig:eg:opt_vals}
\end{figure}

As a direct consequence of Theorem~\ref{opt_values}, we can relate the sets $\mathcal{S}$ and $\mathcal{S}_1$. In particular, under condition~\eqref{reg1}, any solution of~\eqref{optim-1} remains a solution of~\eqref{optim-2}.

\begin{corollary}\label{Sols}
Under the assumptions of Theorem~\ref{opt_values}, $\mathcal{S}\subset\mathcal{S}_1$.
\end{corollary}

\begin{proof} Define $v$ and $\bar v$ by~\eqref{eq:op_val}. If $\bar x \in\mathcal S$, then $\bar x\in D$ and $f(\bar x)=v$. Since $v=\bar v$ by Theorem~\ref{opt_values} and $\bar f(\bar x)\leq f(\bar x)$, this implies that $\bar f(\bar x)\leq \bar v$. As $\bar x\in D$ and $D\subset\bar D$, we see that $\bar x$ is a feasible point of~\eqref{optim-2}. So, the inequality $\bar f(\bar x)\leq \bar v$ forces $\bar f(\bar x)= \bar v$. Consequently, $\bar x$ belongs to $\mathcal{S}_1$. We have thus proved that $\mathcal{S}\subset \mathcal{S}_1$. $\hfill\Box$
\end{proof}

\begin{remark} The solution set of~\eqref{optim-1} can be a proper subset of the solution set of~\eqref{optim-2}. To justify this claim, consider the nearly convex optimization problem~\eqref{optim-1}, where the objective function $f$ and the constraint set $D$ are given in Example~\ref{eg:3.2}. Condition~\eqref{reg1} is satisfied because $\dom f=\R^2$ and $$\ri D={\rm int}\, D=(0,+\infty)\times\R.$$ Since $f$ is a continuous function on $\R^2$, we have $\bar f(x)=f(x)$ for all $x\in\mathbb R$ (see Lemma~\ref{lem:bar_f_lsc}). In addition, as $\bar D=\R_+\times\R$, the solution set of~\eqref{optim-2} is $\{0\}\times \R$. Clearly, the latter contains the solution set $\mathcal S$ of~\eqref{optim-1}, which was described in Example~\ref{eg:3.2}, as a proper subset. 	
\end{remark}

 It is of interest to determine under which conditions a point $\bar x \in \mathcal{S}_1$ also belongs to $\mathcal{S}$. The following corollary describes one situation where this occurs.

\begin{corollary}\label{Sols_1}
Under the assumptions of Theorem~\ref{opt_values}, any point $\bar x \in \mathcal{S}_1$ belonging to the set $D \cap {\rm ri}(\dom \bar f)=D \cap {\rm ri}(\dom f)$ is a solution of~\eqref{optim-1}.  In other words, 
\begin{equation}\label{reverse_incl}
\mathcal{S}_1\cap D \cap {\rm ri}(\dom \bar f)=\mathcal{S}_1\cap D \cap {\rm ri}(\dom f)\subset \mathcal S.\end{equation}
\end{corollary}

\begin{proof}  First, recall that ${\rm ri}(\dom \bar f)={\rm ri}(\dom f)$ (see Lemma~\ref{eq_in_ridom}). Hence,  $$D \cap {\rm ri}(\dom \bar f)=D \cap {\rm ri}(\dom f).$$ Take any $\bar x \in \mathcal{S}_1$ such that $\bar x \in D \cap {\rm ri}(\dom \bar f)$. Since condition~\eqref{reg1} holds, applying Theorem~\ref{opt_values} yields
$$
v:=\inf\{f(x)\mid x\in D\}=\inf\{\bar f(x)\mid x\in \bar{D}\}.
$$
As $\bar x \in {\rm ri}(\dom \bar f)$, invoking Lemma~\ref{eq_in_ridom} gives $f(\bar x)=\bar f(\bar x)=v$. Moreover, since $\bar x \in D$, it follows that $\bar x\in\mathcal S$.
	$\hfill\Box$
	\end{proof}

In connection with the inclusion~\eqref{reverse_incl}, it is natural to ask:  \textit{What happens if $\bar x \notin {\rm ri}(\dom f)$ or, equivalently, whether the implication
$$
\bar x \in \mathcal{S}_1 \;\Longrightarrow\; \bar x \in \mathcal{S}
$$
still holds if only $\bar x\in D$?} The answer is negative, as shown by the next example.

\begin{example}\label{example_3.4}
Let $ C=\big([0,1]\times[0,1]\big)\setminus\{(0,0)\}.$ Clearly, $C$ is a convex set. Consider the function $f:\mathbb R^2\to\overline{\mathbb R}$ defined by
$$
f(x)=
\begin{cases}
0, & \text{if } x\in C,\\
 1 & \text{if } x=(0,0),\\
+\infty, & \text{otherwise},
\end{cases}
$$
and let $D=[0,1]\times[0,1]$. We observe that $
\ri D\cap \ri(\dom f)=(0,1)\times(0,1),$ and hence condition~\eqref{reg1} is satisfied. Since $$\epi f=\big(C\times[0,+\infty)\big)\cup \big(\{(0,0)\}\times [1,+\infty)\big)$$ is a convex set, $f$ is a convex function. The function $\bar f$ is given by
$$
\bar f(x)=
\begin{cases}
0, & \text{if } x\in D,\\
+\infty, & \text{if } x\notin D.
\end{cases}
$$
Clearly, $\bar x:=(0,0)\in D$ is a solution of~\eqref{optim-2}. However,
since $f(\bar x)=1$ and $f(x)=0$ for all $x\in D\setminus\{\bar x\}$, the point $\bar x$ is not a solution of~\eqref{optim-1}.
\end{example}

\subsection{Local solutions and global solutions}

It is well known that any local minimizer of a convex optimization problem is also a global minimizer. Hence, the next question is meaningful.

 \smallskip
\textbf{Question 3.} \textit{Does the set of local solutions of a nearly convex optimization problem coincide with the set of its global solutions? If this is not true in general, under what conditions does the set of local solutions coincide with the set of global solutions?}

  \smallskip
  Bo{\c{t}}, Grad, and Wanka~\cite[Theorem~2.3]{bkw2008} have proved the coincidence of the local solution set and the global solution set for any unconstrained optimization problem having a proper nearly convex objective function. (Observe that the notion of almost convexity in \cite[Definition~2.1]{bkw2008} coincides with the notion of near convexity in Definition~\ref{def_Minty}.)

For a constrained nearly convex optimization problem, the set of local solutions can be strictly larger than the set of its global solutions.

\begin{example}\label{local_global_sols} Consider the constrained nearly convex optimization problem of the form~\eqref{optim-1} with $f$ and $D$ being given as in Example~\ref{eg:opt_vals}. This problem has two local solutions: $\bar x:=\Big(0,\dfrac{1}{4}\Big)$ and $\hat x:=\Big(0,\dfrac{2}{3}\Big)$. Among the two points, $\hat x$ is a unique global solution. Note that, since $f(\bar x)=\dfrac{1}{4}$ and $f(\hat x)=\dfrac{1}{6}$, $\bar x$ is a unique local non-global solution of~\eqref{optim-1}.
\end{example}

Example~\ref{local_global_sols} has solved the first half of Question~3 in the negative. The forthcoming theorem provides us with an answer to the second half of Question~3. We will prove the theorem in two ways: The first one is to use Theorem~\ref{opt_values}, and the second one is to derive the result from the above-mentioned theorem of~\cite{bkw2008} and a sum rule for nearly convex functions.

\begin{theorem}\label{local_sols}
If~\eqref{optim-1} is a nearly convex optimization problem and the regularity condition~\eqref{reg1} is fulfilled, then any local solution $\bar x$ of~\eqref{optim-1} with $f(\bar x)\in\R$ is a global solution. Thus, under the regularity condition~\eqref{reg1}, the set of the local solutions of~\eqref{optim-1} having finite values of the objective function coincides with the global solution set.
\end{theorem}

\noindent {\it First proof} (based on Theorem~\ref{opt_values} and a sum rule)  Define the indicator function of $D$ by setting
\begin{equation}\label{indicator}
\delta_D(x)=\begin{cases} 0 & \ \; \mbox{\rm if }\; x\in D\\
	+\infty & \ \;  \mbox{\rm if }\; x\notin D.\end{cases}
\end{equation} Put \begin{equation*}\label{sum_function} \varphi(x)=f(x)+\delta_D(x)\quad\; (x\in\R^n).\end{equation*} Let $\bar x$ be a local solution of~\eqref{optim-1} with $f(\bar x)\in\R$. Then, $\bar x\in (\dom f)\cap D$ and there is $\varepsilon>0$ such that $f(x)\geq f(\bar x)$ for all $x\in \bar B(\bar x,\varepsilon)\cap D$. Hence, $\bar x$ is a global solution of the optimization problem
\begin{equation}\label{optim-varphi}\min\Big\{\varphi(x)\mid x\in \bar B(\bar x,\varepsilon)\Big\}.
\end{equation} As $D$ is a nearly convex set, $\epi\delta_D=D\times\R_+$ is also a nearly convex set; so~$\delta_D$ is a nearly convex function. Condition~\eqref{reg1} implies that
\begin{equation}\label{reg_2}
\big({\rm ri}(\dom f)\big)\cap \big ({\rm ri}(\dom \delta_D)\big)= \big({\rm ri}(\dom f)\big)\cap \big(\ri D\big)\neq\emptyset.
\end{equation} Therefore, by the sum rule in Proposition~\ref{nc sum f}, $\varphi=f+\delta_D$ is a proper nearly convex function. Fix a point $x_0\in \big({\rm ri}(\dom f)\big)\cap \big(\ri D\big).$ Since $\bar x\in\dom f$ and $\dom f$ is a nearly convex set (see Proposition~\ref{nc dom}), by Proposition~\ref{Connection} we can assert that the point $x_t:=(1-t)\bar x+tx_0$ belongs to ${\rm ri}(\dom f)$ for any $t\in (0,1]$. Similarly, as $\bar x\in D$ and $x_0\in\ri D$, the point $x_t$ is contained in $\ri D$ for any $t\in (0,1]$. So, if $t\in (0,1)$ is small enough, then $$x_t\in \big({\rm ri}(\dom f)\big)\cap \big ({\rm ri}(\dom \delta_D)\big)\cap B(\bar x,\varepsilon).$$ In addition, since $\dom\varphi=(\dom f)\cap (\dom \delta_D)$, using~\eqref{reg_2} and Proposition~\ref{T1}(b) gives ${\rm ri}(\dom\varphi)=\big({\rm ri}(\dom f)\big)\cap \big ({\rm ri}(\dom \delta_D)$. Hence, if $t\in (0,1)$ is small enough, then $x_t\in {\rm ri}(\dom\varphi)\cap B(\bar x,\varepsilon)$. In particular, 
\begin{equation*}{\rm ri}(\dom\varphi)\cap B(\bar x,\varepsilon)\neq\emptyset.\end{equation*} This means that the nearly convex optimization problem~\eqref{optim-varphi} satisfies the regularity condition of the type~\eqref{reg1}. Therefore, by Theorem~\ref{opt_values} we have 
$$f(\bar x)=\min\Big\{\varphi(x)\mid x\in \bar B(\bar x,\varepsilon)\Big\}=\min\Big\{\bar\varphi(x)\mid x\in \bar B(\bar x,\varepsilon)\Big\},$$ where the function $\bar\varphi$ corresponding to $\varphi$ is defined in accordance with~\eqref{bar_f}. Since \begin{equation}\label{ineq_n1}\bar\varphi(\bar x)\leq\varphi(\bar x)=f(\bar x)+\delta_D(\bar x)=f(\bar x),\end{equation} this yields $\bar\varphi(\bar x)\leq \bar\varphi(x)$ for all $x\in \bar B(\bar x,\varepsilon)$. This shows that $\bar x$ is a local solution of the optimization problem
\begin{equation}\label{optim-bar-varphi}\min\Big\{\bar\varphi(x)\mid x\in \bar B(\bar x,\varepsilon)\Big\}.
\end{equation} As $\varphi$ is a proper nearly convex function, $\bar\varphi$ is a proper convex function by Lemma~\ref{lem:bar_f}. Then,~\eqref{optim-bar-varphi} is a convex optimization problem. It follows that $\bar x$ is a global solution of~\eqref{optim-bar-varphi}. Since ${\rm ri}(\dom\varphi)\neq\emptyset$, by Theorem~\ref{opt_values} we have 
\begin{equation}\label{values_equality}\min\Big\{\varphi(x)\mid x\in\R^n\Big\}=\min\Big\{\bar\varphi(x)\mid x\in\R^n\Big\}= \bar\varphi(\bar x).\end{equation} If 
\begin{equation}\label{eq_bar_x}\bar\varphi(\bar x)=f(\bar x),\end{equation} then the relations in~\eqref{values_equality} show that $\bar x$ is a global solution of~\eqref{optim-1}. If~\eqref{eq_bar_x} were false, then by~\eqref{ineq_n1} one would have \begin{equation}\label{ine_strict}\bar\varphi(\bar x)<f(\bar x).\end{equation}
Since $\varphi$ is a proper nearly convex function, $\bar\varphi$ is a proper convex function by Lemma~\ref{lem:bar_f}. Hence, as $f(\bar x)\in\R$, from~\eqref{ine_strict} we get $\bar\varphi(\bar x)\in \R$. In addition, according to Lemma~\ref{lem:bar_f_lsc},
\begin{equation}\label{bar_varphi}
\bar\varphi(\bar x)=\liminf\limits_{x\to\bar x}\varphi(x). 
\end{equation} Let $\{x_k\}\subset\R^n$ be a sequence with $\lim\limits_{k\to\infty} x_k=\bar x$ and \begin{equation}\label{x_k}\lim\limits_{k\to\infty} \varphi(x_k)=\liminf\limits_{x\to\bar x}\varphi(x).\end{equation} Recalling that $\varphi(x_k)=f(x_k)+\delta_D(x_k)$, by~\eqref{bar_varphi}, the properness of $f$, and the inclusion $\bar\varphi(\bar x)\in \R$, we can assume without any loss of generality that $\delta_D(x_k)=0$ for all $k\in\mathbb N$. Then, combining~\eqref{x_k} with~\eqref{bar_varphi} yields 
$\lim\limits_{k\to\infty} f(x_k)=\bar\varphi(\bar x).$ So, by~\eqref{ine_strict} we have $\lim\limits_{k\to\infty} f(x_k)<f(\bar x).$ Hence, there exists $\bar k\in\mathbb N$ such that $f(x_k)<f(\bar x)$ and $x_k\in \bar B(\bar x,\varepsilon)\cap D$ for all $k\geq\bar k$. This contradicts the local optimality of $\bar x$. 

Summing up all the above, we conclude that $\bar x$ is a global solution of~\eqref{optim-1}. $\hfill\Box$

\smallskip
\noindent {\it Second proof} (based on Theorem~2.3 from~\cite{bkw2008} and a sum rule) Define $\varphi$ by~\eqref{indicator}. Since condition~\eqref{reg1} implies~\eqref{reg_2}, by the sum rule in Proposition~\ref{nc sum f} we see that $\varphi=f+\delta_D$ is a proper nearly convex function. As $\bar x$ is a local solution of~\eqref{optim-1} with $f(\bar x)\in\R$, it is a local solution of the unconstrained nearly convex optimization problem \begin{equation}\label{unconst_varphi} \min\Big\{\varphi(x)\mid x\in \mathbb R\Big\}.\end{equation} Applying Theorem~2.3 from~\cite{bkw2008}, we can assert that $\bar x$ is a global solution of~\eqref{unconst_varphi}. Since $\varphi(x)=f(x)+\delta_D(x)$ for all $x\in\R^n$, this implies that $\bar x$ is a global solution of~\eqref{optim-1}.
$\hfill\Box$

\smallskip
The condition $f(\bar x)\in\R$ is essential for the validity of the first assertion of Theorem~\ref{local_sols}. To justify this claim, let us consider the following example.

\begin{example}\label{local_global_2} Consider the constrained nearly convex optimization problem of the form~\eqref{optim-1} where $D$ is the same as in Example~\ref{eg:opt_vals} and $$f(x)=\begin{cases} \left|x_2-\dfrac{1}{2}\right| \quad&  \mbox{\rm if }\; x=(x_1,x_2)\in \R \times \left(-\infty,\dfrac{3}{4}\right],\\
		+\infty & \mbox{\rm otherwise}.\end{cases}$$ Clearly, this problem satisfies the regularity condition~\eqref{reg1}. A direct verification shows that each point $\bar x=(\bar x_1,\bar x_2)\in D$ with $\bar x_1>\dfrac{3}{4}$, where $f(\bar x)=+\infty$, is a local non-global solution. Note that the optimal value 0 of the problem is attained at every point from the set $(0,1]\times\big\{\dfrac{1}{2}\big\}$.   		
\end{example}

\section{Optimality Conditions}\label{optimality}

In this section, we first obtain optimality conditions for nearly convex optimization problem under geometrical constraints, focusing on the Fermat rules and the relationships between optimality conditions for the given problem and its associated convex problem. Then, we establish optimality conditions for  nearly convex optimization problems under geometrical constraints and functional constraints in the form of Lagrange multiplier rules and Karush-Kuhn-Tucker conditions. As in the preceding section, the solution sets of~\eqref{optim-1} and~\eqref{optim-2} are denoted, respectively, by~$\mathcal{S}$ and~$\mathcal{S}_1$.

\subsection{Fermat's rules}\label{Fermat_subsection}

Consider the optimization problem in the form~\eqref{optim-1} and let $\delta_D(x)$ be defined by~\eqref{indicator}.  It is clear that a vector $\bar x\in D$ is a solution of~\eqref{optim-1} if and only if
 $$(f+\delta_D)(\bar x)\leq (f+\delta_D)(x)\ \ \mbox{\rm for all }\, x\in\mathbb R^n.$$ By Definition~\ref{subdifferential}, the latter can be rewritten equivalently as
 \begin{equation}\label{optimality-1a}
 	0\in\partial (f+\delta_D)(\bar x).
 \end{equation}  It is well known (see, e.g., the proof of Theorem~5.1 in~\cite{nty1}) that the set $D$ is nearly convex if and only if its indicator function $\delta_D$ is nearly convex. Recall that the normal cone to a nearly convex set at a point belonging to the set has been defined in Definition~\ref{normal_cone}.

 The following result, which directly employs the data of~\eqref{optim-1}, extends and deepens the result in~\cite[Proposition~6.1]{LM2019}.
 
 \begin{theorem}\label{Fermat} {\rm \textbf{(Fermat's rule for nearly convex optimization problems)}} A point $\bar x\in D$ is a solution of the nearly convex optimization problem~\eqref{optim-1} if and only if the inclusion~\eqref{optimality-1a} holds. Moreover, under the regularity condition~\eqref{reg1}, a point $\bar x\in D$ is a solution of~\eqref{optim-1} if and only if
 	\begin{equation}\label{optimality-1b}
 		0\in\partial f(\bar x)+  N(\bar x;D);
 	\end{equation}
 	 hence
 	$\mathcal{S}=\{\bar x\in\R^n\mid 0\in\partial f(\bar x)+  N(\bar x;D)\}.$
 \end{theorem}
\begin{proof}
The first assertion has already been proved above. The second assertion follows because, under the assumptions of the theorem, we  can apply the subdifferential sum rule in Proposition~\ref{sum_rule} to the nearly convex functions $f$ and $\delta_D$ to derive~\eqref{optimality-1b} from~\eqref{optimality-1a}  and the formula $\partial\delta_D(\bar x)=N(\bar x;D)$.  $\hfill\Box$
\end{proof}

The next example shows how Theorem~\ref{Fermat} can be used in practice.

\begin{example}\label{example_4.1}
	Consider problem~\eqref{optim-1}, where $f$ and $D$ are the same as in Example~\ref{example_3.4}. Recall that the regularity condition~\eqref{reg1} is satisfied. It is easy to verify that for any $\bar x\in D\setminus\{(0,0)\}$, $\partial f(\bar x) = \delta_D(\bar x)=N(\bar x; D)$. So, one has
	$$\partial f(\bar x)+  N(\bar x;D)=N(\bar x;D)+ N(\bar x;D) = N(\bar x;D).$$
	Since $0\in N(\bar x;D)$, the inclusion~\eqref{optimality-1b} holds. By Theorem~\ref{Fermat}, this implies that $(D\setminus\{(0,0)\})\subset\mathcal S$.  For $\bar x = (0,0)$, as $\partial f(\bar x)=\emptyset$, the inclusion~\eqref{optimality-1b} fails. Therefore, we get $	\mathcal{S}=D\setminus\{(0,0)\}$.
\end{example}
 
 In general, the calculation of the subdifferential of a nearly convex function is more complicated than that of a convex function. So, solving problem~\eqref{optim-2} is  often easier than solving problem~\eqref{optim-1}. The forthcoming theorem paves a way to solve the nearly convex optimization problem~\eqref{optim-1} by using Fermat's rule for its associated convex problem~\eqref{optim-2}.
 
  \begin{theorem}\label{Fermat1} {\rm \textbf{(Fermat's rule for the associated convex problem and its consequences)}} For the nearly convex optimization problem~\eqref{optim-1}, the following assertions are valid.
  		\begin{itemize}
  			\item[{\rm (a)}] A point $\bar x\in \bar D$ is a solution of the convex optimization problem~\eqref{optim-2} if and only if 
  			\begin{equation}\label{optimality-2a}
  				0\in\partial (\bar f+\delta_{\bar D})(\bar x).
  			\end{equation}
  			\item[{\rm (b)}]  {\rm \textbf{(Fermat's rule for the associated convex problem)}} If the regularity condition~\eqref{reg1} is satisfied, then a point $\bar x\in \bar D$ is a solution of~\eqref{optim-2} if and only if
  			\begin{equation}\label{optimality-2b}
  				0\in\partial \bar f(\bar x)+  N(\bar x;\bar D);
  			\end{equation}
  			hence
  			\begin{equation}\label{sol_set_S1}
  				\mathcal{S}_1=\{\bar x\in\R^n\mid 0\in\partial \bar f(\bar x)+  N(\bar x;\bar D)\}.
  			\end{equation}
  			\item[{\rm (c)}] {\rm \textbf{(Necessary optimality condition for a  nearly convex optimization problem)}} If the regularity condition~\eqref{reg1} holds and $\bar x\in D$ is a solution of~\eqref{optim-1}, then the inclusion~\eqref{optimality-2b} holds.
  			\item[{\rm (d)}] {\rm \textbf{(Sufficient optimality condition for a  nearly convex optimization problem)}} If the regularity condition~\eqref{reg1} is fulfilled, then any point $\bar x \in D \cap {\rm ri}(\dom \bar f)$, which satisfies the inclusion~\eqref{optimality-2b}, is a solution of~\eqref{optim-1}.
  		\end{itemize}
 \end{theorem}
 
  \begin{proof} Assertion~(a) follows from the definition of subdifferential given recalled in Section~\ref{nc_problems}. 
  	
  	To prove~(b), we get by the near convexity of $f$ that ${\rm ri} (\dom \bar f)={\rm ri}(\dom f)$ (see Lemma~\ref{eq_in_ridom}). Then, from the near convexity of $D$ we can deduce that 
  	$${\rm ri}\big(\dom\big(\delta_{\bar D}\big)\big)=\ri(\bar D)=\ri D.$$ Therefore, thanks to the regularity condition~\eqref{reg1}, we have
  	$$\big({\rm ri}(\dom\bar f)\big)\cap \Big({\rm ri}\big(\dom\big(\delta_{\bar D}\big)\big)\Big)=\big({\rm ri}(\dom f)\big)\cap (\ri D)\neq\emptyset.$$ So, by the subdifferential sum rule in Proposition~\ref{sum_rule} we have for every $\bar x\in\bar D$ the following
  	$$\partial (\bar f+\delta_{\bar D})(\bar x)=\partial\bar f(\bar x)+ \partial\delta_{\bar D}(\bar x)=\partial \bar f(\bar x)+  N(\bar x;\bar D).$$
  	Hence, the inclusion~\eqref{optimality-2a} can be rewritten equivalently as the one in~\eqref{optimality-2b}. Now, it is clear that assertion~(b) follows from assertion~(a).
  	
  	Assertion~(c) follows from Corollary~\ref{Sols} and assertion~(b), while 
  	assertion~(d) is a direct consequence of Corollary~\ref{Sols_1} and assertion~(b).
  	$\hfill\Box$
  \end{proof}

\begin{example}\label{example_4.2}
	Consider problem~\eqref{optim-1} with $f$ and $D$ defined in Example~\ref{example_3.4}. Since the regularity condition~\eqref{reg1} is satisfied, formula~\eqref{sol_set_S1} is valid. As $\bar D=D=[0,1]\times[0,1]$ and 
	$$
	\bar f(x)=
	\begin{cases}
		0, & \text{if } x\in D,\\
		+\infty, & \text{if } x\notin D,
	\end{cases}
	$$ we see that $\mathcal{S}_1=D$. This equality is in full agreement with~\eqref{sol_set_S1}. Note that assertion~(d) in Theorem~\ref{Fermat1} assures that $(0,1)\times (0,1)\subset\mathcal{S}$. Since  $	\mathcal{S}=D\setminus\{(0,0)\}$, the last inclusion is strict.
\end{example}
 
\subsection{Lagrange multiplier rules}\label{Lagrange_subsection}
 Usually, the constraint set of~\eqref{optim-1} is of the form 
 \begin{equation}\label{D_intersection}
 	D=\Omega_0\cap\Omega_1,
 \end{equation}
 where $\Omega_0$ is a nonempty set in $\R^n$ and  \begin{equation}\label{constraint}
 	\Omega_1=\big\{x\in\R^n\mid g_i(x)\leq 0 \ \,\mbox{for each }\, i=1, \ldots, m\big\},
 \end{equation}
 with $g_i\colon \R^n \to \oR$, $i=1, \ldots, m$, being proper functions. The set $\Omega_0$ is called a \textit{geometrical constraint }and the inequalities  $g_i(x)\leq 0,$ $i=1, \ldots, m$, are said to be \textit{functional constraints}.
 
 \begin{definition}\label{nc_prob2} {\rm We say that~\eqref{optim-1} is a \textit{nearly convex optimization problem under a geometrical constraint and functional constraints} if $f$ is nearly convex and the set~$D$ defined by~\eqref{D_intersection} is nearly convex. In that case, if the functions $g_1,\ldots,g_m$ are absent, then we have a \textit{nearly convex optimization problem under a geometrical constraint}. If $\Omega_0=\R^n$, then we have a \textit{nearly convex optimization problem under functional constraints}.}
 \end{definition}
 
 \begin{definition}\label{Slater} {\rm  If there is a point $x^0\in \big(\ri\Omega_0\big)\cap \left(\bigcap\limits_{i=1}^m\mbox{\rm ri}(\dom g_i)\right)$ with $g_i(x^0) < 0$ for all $i=1, \ldots, m$, then we say that the problem~\eqref{optim-1}, where $D$ is given by~\eqref{D_intersection}, satisfies the \textit{generalized Slater condition}.}
 \end{definition}
 
 The generalized Slater condition is a kind of \textit{constraint qualification}, also called a \textit{regularity condition}, for the optimization problem under consideration. It is important for studying convex optimization problems, as well as nearly convex optimization problems.
 
 The following lemma gives a sufficient conditions for an optimization problem under a geometrical constraint and functional constraints to be nearly convex.
 
 \begin{lemma}\label{nc_funct} Suppose that $f$ is a proper nearly convex function, $\Omega_0$ is a nearly convex set, and  $g_1,\ldots, g_m$  are proper nearly convex functions. If the problem~\eqref{optim-1} with~$D$ being given by~\eqref{D_intersection} satisfies the generalized Slater condition, then it is a nearly convex optimization problem.
 \end{lemma}
 \begin{proof}
  According to Definition~\ref{nc_prob1},  we need only to prove that  the set $D$ defined by~\eqref{D_intersection} is nearly convex. By the generalized Slater condition, we can find a point
  \begin{equation}\label{x0}
  	x^0\in \big(\ri\Omega_0\big)\cap \left(\bigcap\limits_{i=1}^m\mbox{\rm ri}(\dom g_i)\right)
  \end{equation} such that  $g_i(x^0) < 0$ for all $i=1, \ldots, m$.  	Define
 	\begin{equation}\label{MF-1}
 		g(x)=\max\big\{g_i(x)\mid  i=1, \ldots, m\big\},\ \; x\in \R^n.
 	\end{equation} 
 	Since $g_1,\ldots, g_m$ are nearly convex functions and by~\eqref{x0} one has $\bigcap\limits_{i=1}^m\mbox{\rm ri}(\dom g_i)\neq\emptyset$, it follows from Proposition~\ref{NCMf} that $g$ is a nearly convex function. Moreover, by Proposition~\ref{nc dom}, the sets $\dom g_1,\ldots, \dom g_m$ are nearly convex. Therefore, thanks to the condition $\bigcap\limits_{i=1}^m\mbox{\rm ri}(\dom g_i)\neq\emptyset$, we can apply Proposition~\ref{T1}(b) to the sets $\dom g_1,\ldots,\dom g_m$ to obtain
 	\begin{equation*}
 		\ri\left(\bigcap_{i=1}^m\dom g_i\right)=\bigcap_{i=1}^m\mbox{\rm ri}(\dom g_i).
 	\end{equation*} 
 	Since $\dom g=\bigcap\limits_{i=1}^m\dom g_i$, this yields
  		${\rm ri}(\dom g)=\bigcap\limits_{i=1}^m\mbox{\rm ri}(\dom g_i)$.
 Hence, the inclusion~\eqref{x0} implies that $x^0\in{\rm ri}(\dom g)$.  
 	
  Next, by~\eqref{MF-1} we can represent the set~$\Omega_1$ in~\eqref{constraint} as
 	\begin{equation}\label{constraint-1} \Omega_1= \big\{x\in\R^n\mid g(x)\leq 0\big\}.
 	\end{equation}
 		Define the linear mapping $A:\R^n\times\R\to\R^n$ by setting $A(x,\mu)=x$ for every $(x,\mu)\in \R^n\times\R$. Then, by~\eqref{D_intersection} and~\eqref{constraint-1} one gets
 			\begin{equation}\label{constraint-2} D=A\Big(\big(\Omega_0\times (-\infty,0]\big)\cap  (\epi g)\Big).
 		\end{equation}
 		Since $g_i(x^0) < 0$ for $i=1, \ldots, m$, we have $g(x^0)<0$.  On one hand, for an arbitrarily chosen value $\lambda_0\in \big(g(x^0),0\big)$, by~\eqref{x0}) and Proposition~\ref{ri_prod} we have \begin{equation}\label{ri2} (x^0,\lambda_0)\in \ri \Omega_0\times (-\infty,0)=  \mbox{\rm ri}\, \big(\Omega_0\times (-\infty,0]).\end{equation} 
 		On the other hand, applying Proposition~\ref{riepi} to the nearly convex function $g$, we have
 		\begin{equation*}
 			\mbox{\rm ri}(\epi g)=\Big\{(x, \lambda)\in \R^n\times \R\mid x\in \mbox{\rm ri}(\dom g),\ \lambda>g(x)\Big\}.
 		\end{equation*} 
 	Since $x^0\in {\rm ri}(\dom g)$ and $\lambda_0>g(x^0)$, this equality tells us that
 		\begin{equation}\label{ri1}
 			(x^0,\lambda_0)\in\mbox{\rm ri}(\epi g).
 		\end{equation} 
 		Hence, combining~\eqref{ri2} and~\eqref{ri1} yields
 		\begin{equation}\label{Slater_1} (x^0,\lambda_0)\in \mbox{\rm ri}\, \Big(\Omega_0\times (-\infty,0]\Big)\cap \Big(\mbox{\rm ri}(\epi g)\Big).
 		\end{equation}
 		Therefore, by Proposition~\ref{NCI}, the set $\Big(\Omega_0\times (-\infty,0]\Big)\cap (\epi g)$ is nearly convex. This allows us to apply Proposition~\ref{T1}(a) together with~\eqref{constraint-2} to conclude that the set~$D$ is a nearly convex. $\hfill\Box$
  \end{proof}
  
 Under suitable conditions involving near convexity and regularity, the topological closure of the set $\Omega_1$ in~\eqref{constraint} can be computed via the functions $\bar g_1,\ldots, \bar g_m$.
  
\begin{lemma}\label{lem4.2}
  		Let $g_i:\R^n\to\oR$, $i=1,\ldots, m$, be proper nearly convex functions  such that there exists a point $x^0\in \bigcap\limits_{i=1}^m\mbox{\rm ri}(\dom g_i)$ with $g_i(x^0)< 0$ for $i=1,\ldots, m$. Then,
  		\begin{equation}\label{cl_omega_1}
  			\overline{\Omega_1}= \big\{x\in\R^n\mid \bar g_i(x)\leq 0 \ \,\mbox{for }\, i=1, \ldots, m\big\},
  		\end{equation} where $\Omega_1$ is defined by~\eqref{constraint}.
  	\end{lemma}
    \begin{proof} Let $A:\R^n\times\R\to\R^n$ be the linear mapping considered in the proof of Lemma~\ref{nc_funct}, that is $A(x,\mu)=x$ for every $(x,\mu)\in \R^n\times\R$. Put $I=\{1,\ldots,m\}$. From~\eqref{constraint} it follows that
    	\begin{equation}\label{Omega1_epi_A} \Omega_1= A\left(\Big(\bigcap_{i\in I} \epi g_i\Big)\cap  \big(\R^n\times (-\infty,0]\big)\right).
       	\end{equation}
  		 Since $x^0\in\mbox{\rm ri}(\dom g_i)$ and $g_i(x^0)< 0$ for all $i\in I$, by setting $$\lambda=\max\big\{g_i(x^0)\mid i\in I\big\}$$ and invoking the second assertion of Proposition~\ref{riepi} we have by the near convexity of $g_1,\ldots,g_m$ that $$\big(x^0,\lambda\big)\in {\rm ri}(\epi g_i)\cap\big(\R^n\times (-\infty,0)\big)$$ for all $i\in I$. Hence,
  		\begin{equation}\label{reg_n1}
  		\left(\bigcap_{i\in I}{\rm ri}(\epi g_i)\right)\cap\big(\R^n\times (-\infty,0)\big)\neq \emptyset.
  			\end{equation}
  		 As $\epi g_1,\ldots, \epi g_m$, and $\R^n\times (-\infty,0]$ are nearly convex sets, by the property~\eqref{reg_n1} we can apply Proposition~\ref{T1}(b) to get
  		\begin{equation}\label{cl_epi_epi_cl}
  			\overline{\left(\bigcap_{i\in I} \epi g_i\right)\cap  \big(\R^n\times (-\infty,0]\big)}=\left(\bigcap_{i\in I}\overline{\epi g_i}\right)\cap\big(\R^n\times (-\infty,0]\big).
  		\end{equation}
  		By~\eqref{Omega1_epi_A} and Proposition~\ref{T1}(a), we have
  		$$\overline{\Omega_1}= \overline{A\left(\Big(\bigcap_{i\in I} \epi g_i\Big)\cap  \big(\R^n\times (-\infty,0]\big)\right)}=\overline{A\left(\overline{\Big(\bigcap_{i\in I} \epi g_i\Big)\cap  \big(\R^n\times (-\infty,0]\big)}\right)}.$$
  		Combining this with~\eqref{cl_epi_epi_cl} gives
  		$$\overline{\Omega_1}= \overline{A\left(\Big(\bigcap_{i\in I}\overline{\epi g_i}\Big)\cap\big(\R^n\times (-\infty,0]\big)\right)}.$$
  		It then follows from Lemma~\ref{lem:bar_f} that
  		\begin{equation}\label{bar_Omega1}\begin{array}{rcl}
  		\overline{\Omega_1}&=&\overline{A\left(\Big(\bigcap_{i\in I} \epi \bar g_i\Big)\cap\big(\R^n\times (-\infty,0]\big)\right)}\\
  			&=&\overline{\big\{x\in\R^n\mid \bar g_i(x)\leq 0\ \,\mbox{for }\, i\in I\big\}}.
  			\end{array}
  		\end{equation}
  		 Since $\bar g_1,\ldots,\bar g_m$ are lower semicontinuous functions, the set $$\big\{x\in\R^n\mid \bar g_i(x)\leq 0\ \,\mbox{for }\, i\in I\big\},$$ which is the intersection of $m$ sublevel sets of those functions, is closed. Hence,  from~\eqref{bar_Omega1} we obtain~\eqref{cl_omega_1}. $\hfill\Box$
  	  \end{proof}

  \begin{lemma}\label{lem4.3}
  	Under the assumptions of Lemma~\ref{nc_funct}, if the problem~\eqref{optim-1} with~$D$ being given by~\eqref{D_intersection} satisfies the generalized Slater condition stated in Definition~\ref{Slater}, then 
  		\begin{equation}\label{barD_cap}
  			\bar D=\overline{\Omega_0}\cap\overline{\Omega_1}
  		\end{equation} and, for every $\bar x\in\bar D$,
  		\begin{equation}\label{barD_normal}
  		 N(\bar x;\bar D)= N\big(\bar x;\overline{\Omega_0}\big)+N\big(\bar x;\overline{\Omega_1}\big).
  		\end{equation}
  \end{lemma}
  \begin{proof} Let $g$ be defined as in~\eqref{MF-1}, and let $A:\R^n\times\R\to\R^n$ be the linear mapping introduced in the proof of Lemma~\ref{nc_funct}. By~\eqref{constraint-1}, we have
  		\begin{equation}\label{rep_Omega1}
  		\Omega_1
  		=
  		A\Big((\epi g)\cap\big(\R^n\times (-\infty,0]\big)\Big).
  		\end{equation}
  		As shown in the proof of Lemma~\ref{nc_funct}, the function $g$ is nearly convex. Consequently, the set
  		$(\epi g)\cap\big(\R^n\times (-\infty,0]\big)$,
  		being the intersection of nearly convex sets and the condition  $\Big(\mbox{\rm ri}(\epi g)\Big)\cap \mbox{\rm ri}\, \Big(\Omega_0\times (-\infty,0]\Big)\neq\emptyset$ is satisfied (see~\eqref{Slater_1}), is itself nearly convex. Applied to~\eqref{rep_Omega1}, Proposition~\ref{T1}(a) therefore implies that $\Omega_1$ is nearly convex.
  		
  		Using Proposition~\ref{T1}(a),~\eqref{Slater_1} and~\eqref{rep_Omega1}, one can easily show that $x^0\in \ri \Omega_1$. Hence, $x^0\in \ri \Omega_0 \cap \ri \Omega_1.$
  		Therefore, since both sets $\Omega_0$ and $\Omega_1$ are nearly convex, we have
  		\begin{equation}\label{regu_2}
  			\ri \overline{\Omega_0}\cap \ri \overline{\Omega_1}
  			=
  			\ri \Omega_0\cap \ri \Omega_1
  			\neq \emptyset.
  		\end{equation}
  	 Hence, by Proposition~\ref{T1}(b) we get
  		$$\bar D=\overline{\Omega_0\cap\Omega_1}=\overline{\Omega_0}\cap\overline{\Omega_1}.$$ This shows that the equality~\eqref{barD_cap} is valid. Moreover, thanks to~\eqref{regu_2}, we can apply Proposition~\ref{NCI} to  deduce~\eqref{barD_normal}, where $\bar x\in\bar D$ is arbitrarily chosen, from~\eqref{barD_cap}.
  		   $\hfill\Box$
  	\end{proof}

   To obtain an explicit formula for the normal cone to $\overline{\Omega_1}$ at a point belonging to that set, we will use Proposition~\ref{NCMf} and the following classical result on normal cones to sublevel sets of a convex function.
   
   \begin{lemma}\label{lem4.4}\label{normal_sublevel_set} {\rm (See~\cite[Proposition~2, p.~206]{IT_1979})} Let $\psi:\R^n\to\bar\R$ be a proper convex function and $\bar x\in\R^n$ be such that $\psi(\bar x)=0$. If $\psi$ is continuous at $\bar x$ and there exists a point $x^0$ satisfying $\psi(x^0)<0$, then	$N\left(\bar x;\Omega\right)= {\rm cone}\,\partial \psi(\bar x)$,
   	where $$\Omega=\big\{x\in\R^n\mid \psi(x)\leq 0\big\}.$$ 
   \end{lemma}

\begin{lemma}\label{lem4.5}
	Under the assumptions of Lemma~\ref{lem4.2}, if $\bar x\in \overline{\Omega_1}$ is such that $\bar g_i(\bar x)=0$ for at least one index $i\in\{1,\ldots,m\}$ and the functions $\bar g_1,\ldots,\bar g_m$ are continuous at $\bar x$, then
	\begin{equation}\label{nc_cl_omega_1}
		N\big (\bar x;\overline{\Omega_1}\big )=\left\{\sum_{i\in  I(\ox)}\lambda_ix_i^*\mid \lambda_i\geq 0\ \; {\rm and}\ \; x_i^*\in \partial \bar g_i(\ox)\ \,  {\rm for}\ \, i\in I(\ox)\right\}
	\end{equation}	
	with $I(\ox):=\big\{i=1, \ldots, m\mid \bar g_i(\ox)=0\big\}$ and $\Omega_1$ being defined by~\eqref{constraint}.
\end{lemma}
\begin{proof} Let $\psi:\R^n\to\bar\R$ be the function defined by setting
	 \begin{equation}\label{psi}
		\psi(x)=\max\big\{\bar g_i(x)\mid  i=1, \ldots, m\big\},\ \; x\in \R^n.
	\end{equation} 
	By the assumptions made,  $\psi(\bar x)=0$ and there exists a point $x^0\in \bigcap\limits_{i=1}^m\mbox{\rm ri}(\dom g_i)$ with $\psi(x^0)<0$. Moreover, since $g_1,\ldots, g_m$ are proper nearly convex functions, the functions $\bar g_1,\ldots,\bar g_m$ are convex. Hence, $\psi$ is a proper convex function. 
	
	Since the functions $\bar g_1,\ldots,\bar g_m$ are continuous at $\bar x$, using~\eqref{psi} one can easily show that $\psi$ is continuous at $\bar x$. Clearly, by~\eqref{cl_omega_1} and~\eqref{psi} we have $$\overline{\Omega_1}=\big\{x\in\R^n\mid \psi(x)\leq 0\big\}.$$
	
Applying Proposition~\ref{diff_max} to the maximum function $\psi$ yields
	\begin{equation*}
		\partial\psi (\ox)=\co \big[\bigcup_{i\in I(\ox)}\partial \bar g_i(\ox)\big].
	\end{equation*}
Furthermore, as $\psi$ is continuous at $\bar x$, 	by Lemma~\ref{lem4.4} and the last equality we obtain
		\begin{equation}\label{nc_repre}
			N\left(\bar x;\overline{\Omega_1}\right)= {\rm cone}\,\partial \psi(\bar x)= {\rm cone}\left(\co \big[\bigcup_{i\in I(\ox)}\partial \bar g_i(\ox)\big]\right).
		\end{equation}
	We have
	\begin{equation}\label{cone_co}{\rm cone}\left(\co \big[\bigcup_{i\in I(\ox)}\partial \bar g_i(\ox)\big]\right)= \left\{\sum_{i\in  I(\ox)}\lambda_ix_i^*\mid \lambda_i\geq 0\ \; {\rm and}\ \; x_i^*\in \partial \bar g_i(\ox)\ \,  {\rm for}\ \, i\in I(\ox)\right\}.\end{equation}
	Indeed, take any $x\in {\rm cone}\left(\co \big[\bigcup\limits_{i\in I(\ox)}\partial \bar g_i(\ox)\big]\right)$. Then, one can find  $\lambda \geq 0$ and $y\in \co \big[\bigcup\limits_{i\in I(\ox)}\partial \bar g_i(\ox)\big]$ such that $x=\lambda y$. Since $y\in \co \big[\bigcup\limits_{i\in I(\ox)}\partial \bar g_i(\ox)\big]$, there exist
	 $$p\in \N,\ \ \lambda_k\in [0,1],\ \, k=1,\ldots,p,\ \ \sum_{k=1}^p \lambda_k=1,$$
	 and,  for every $k\in \{1,\ldots,p\}$, $x^*_k\in \partial \bar g_{i_k}(\bar x)$ with $i_k\in I(\bar x)$ , such that  $y= \sum\limits_{k=1}^p\lambda_k x^*_k.$
	 Thus, $$x	=\lambda y=  \sum_{k=1}^p (\lambda\lambda_k) x^*_k.$$ For each $i\in I(\bar x)$, let
	 $K_i= \big\{k\in\{1,\ldots, p\}\mid i_k=i\big\}.$
	 Then, 
	 $$x=  \sum_{i\in I(\bar x)}\sum_{k\in K_i} (\lambda\lambda_k) x^*_k.$$
	 For each $i\in I(\bar x)$, put $\lambda_i = \sum\limits_{k\in K_i} \big(\lambda\lambda_k\big)$. If $\lambda_i=0$,  take any $x^*_i\in \partial \bar g_i(\bar x)$. If $\lambda_i>0$, define $$x^*_i = \dfrac{1}{\lambda_i}\sum_{k\in K_i} (\lambda\lambda_k) x^*_k$$
	 and get by the convexity of $\partial g_i(\bar x)$ that $x^*_i\in \partial \bar g_i(\bar x).$
	 So, $x= \sum\limits_{i\in I(\bar x)}\lambda_i x^*_i$ 	with $\lambda_i \geq 0$ and $x^*_i\in\partial\bar g_i(\bar x)$ for all $i\in I(\bar x)$. This proves that ${\mathcal A}\subset {\mathcal B}$, where ${\mathcal A}$ denotes the left-hand side of~\eqref{cone_co} and ${\mathcal B}$ stands for the right-hand side of~\eqref{cone_co}. To show that ${\mathcal B}\subset {\mathcal A}$, take any $x\in {\mathcal B}$. Then, $x= \sum\limits_{i\in I(\bar x)}\lambda_i x^*_i$ with $\lambda_i\geq 0$ and $x^*_i\in\partial \bar g_i(\bar x)$ for $i\in I(\ox)$. If $x=0$, the inclusion $x\in {\mathcal A}$ is trivial. If $x\neq 0$, then $\lambda := \sum\limits_{i\in I(\bar x)}\lambda_i$ is a positive real number. Set
	 $ y=\sum\limits_{i\in I(\bar x)}\dfrac{\lambda_i}{\lambda} x^*_i$  and
	 observe that $y\in \co \big[\bigcup\limits_{i\in I(\ox)}\partial \bar g_i(\ox)\big]$. Since $x=\lambda y$, the latter implies that $x\in {\mathcal A}$. Hence, ${\mathcal B}\subset {\mathcal A}$. Thus, the identity~\eqref{cone_co} is valid.
	 	 
	 Combining~\eqref{cone_co} with~\eqref{nc_repre}, we obtain~\eqref{nc_cl_omega_1} and complete the proof.
 $\hfill\Box$
\end{proof}

 \begin{theorem}\label{Lagrange_associated} {\rm \textbf{(Kuhn-Tucker conditions for the associated convex optimization problem)}} Suppose that $f$ is a proper nearly convex function, $\Omega_0$ is a nearly convex set, and  $g_1,\ldots, g_m$  are proper nearly convex functions. Assume that the problem~\eqref{optim-1}, where $D$ is given by~\eqref{D_intersection}, satisfies the generalized Slater condition in the sense of Definition~\ref{Slater}, where $x^0\in{\rm ri}(\dom f)$, and the functions $\bar g_1,\ldots, \bar g_m$ are continuous at a point $\bar x\in\bar D$. Then, $\bar x$ is a solution of the associated convex optimization problem~\eqref{optim-2} if and only if there exist Lagrange multipliers $\lambda_1\geq 0,\ldots,\lambda_m\geq 0$, such that
 	\begin{equation}\label{incl_bar}
 	0\in \partial \bar f(\bar x)+\sum\limits_{i=1}^m \lambda_i\partial \bar g_i(\bar x) + N(\bar x;\overline{\Omega_0})
 	\end{equation} and $\lambda_i\bar g_i(\bar x)=0$ for $i=1, \ldots, m$.
 \end{theorem}
  \begin{proof} First, let us deduce from the results and arguments of the proof of Lemma~\ref{nc_funct} some useful facts. By the assumptions made, the inclusion~\eqref{x0} is satisfied and  $g_i(x^0) < 0$ for all $i\in I$, where $I:=\{1, \ldots, m\}$. Using the maximum function $g$ given by~\eqref{MF-1}, we can represent the constraint set $D$ by~\eqref {constraint-2}, that is
  		\begin{equation*}D=A\Big(\big(\Omega_0\times (-\infty,0]\big)\cap  (\epi g)\Big)
  		\end{equation*}
  		with the linear mapping $A:\R^n\times\R\to\R^n$ being given by the formula $A(x,\mu)=x$ for every $(x,\mu)\in \R^n\times\R$. Since $\Omega_0$ is a nearly convex set and $g$ is a nearly convex function,  $\Omega_0\times (-\infty,0]$ and $\epi g$ are nearly convex sets. For an arbitrarily chosen value $\lambda_0\in \big(g(x^0),0\big)$, we have~\eqref{Slater_1}, that is
  		\begin{equation*} (x^0,\lambda_0)\in \mbox{\rm ri}\, \Big(\Omega_0\times (-\infty,0]\Big)\cap \Big(\mbox{\rm ri}(\epi g)\Big).
  		\end{equation*} So, applying Proposition~\ref{T1} to the above representation of $D$ yields not only the near convexity of $D$ but also the inclusion $x^0\in\ri D$. As $x^0\in\dom f$ by our assumptions, it follows that $x^0\in(\ri D)\cap (\dom f)$. So, the regularity condition~\eqref{reg1} is fulfilled.  
  		
  		 Now, it is clear that we can use the Fermat rule for the associated convex problem in Theorem~\ref{Fermat1} to assert that the given point $\bar x\in\bar D$ is a solution of~\eqref{optim-2} if and only if it satisfies the inclusion~\eqref{optimality-2b}.  Furthermore, according to Lemma~\ref{lem4.3}, formulas~\eqref{barD_cap} and~\eqref{barD_normal} hold. Hence, the inclusion~\eqref{optimality-2b} can be rewritten equivalently as
  		 \begin{equation}\label{optimality-2c}
  		 	0\in\partial \bar f(\bar x)+ N\big(\bar x;\overline{\Omega_0}\big)+N\big(\bar x;\overline{\Omega_1}\big).
  		 	\end{equation} 
        There are two situations: (a) $\bar g_i(\bar x)<0$ for all $i\in I$; (b) There exists at least one index $i\in I$ such that $\bar g_i(\bar x)=0$.
  	 	
  	 	If the situation~(a) appears, the continuity of the functions $\bar g_1,\ldots, \bar g_m$ are continuous at a point $\bar x$ implies that $\bar x\in {\rm int}\,\overline{\Omega_1}$. Then, $N\big(\bar x;\overline{\Omega_1}\big)=\{0\}$. Choosing $\lambda_i=0$ for all $i\in I$, we see that~\eqref{optimality-2c} is equivalent to~\eqref{incl_bar}. Note that the complementarity slackness condition, which requires $\lambda_ig_i(\bar x)=0$ for all $i\in I$, holds. Thus, the conclusion of the theorem is valid.
  	 	
  	 	If the situation~(b) occurs, then by Lemma~\ref{lem4.5} we can compute the normal cone $N\big(\bar x;\overline{\Omega_1}\big)$ by formula~\eqref{nc_cl_omega_1}. Set $\lambda_i=0$ for all $i\in I\setminus I(\bar x)$. Then, substituting the right-hand side of~\eqref{nc_cl_omega_1} for $N\big(\bar x;\overline{\Omega_1}\big)$ in~\eqref{optimality-2c}, we see that~\eqref{optimality-2c} is equivalent to~\eqref{incl_bar}. Clearly, the complementarity slackness condition is fulfilled. We have shown that the conclusion of the theorem holds.
  		 
  	  	The proof is complete. 	$\hfill\Box$
  	\end{proof}

 \begin{theorem}\label{Lagrange_NC} {\rm \textbf{(Lagrange multiplier rule for the original nearly convex optimization problem)}}  Suppose that all the assumptions of Theorem~\ref{Lagrange_associated} are satisfied and $\bar x \in D\cap (\dom f)$ is such that $\bar f(\bar x)=f(\bar x)$ and $\bar g_i(\bar x)=g_i(\bar x)$ for all $i=1,\ldots,m$. Then, $\bar x$ is a solution of the optimization problem under functional constraints given by~\eqref{optim-1} and~\eqref{D_intersection}, then there exist Lagrange multipliers $\lambda_1\geq 0,\ldots,\lambda_m\geq 0$, such that
 		\begin{equation}\label{incl_original} 0\in \partial f(\bar x)+\sum\limits_{i=1}^m \lambda_i\partial g_i(\bar x)  + N(\bar x;\Omega_0)\end{equation} and $\lambda_ig_i(\bar x)=0$ for $i=1, \ldots, m$. 
 	Conversely, if the given point $\bar x$ belongs to ${\rm ri}(\dom f)$ and there exist Lagrange multipliers $\lambda_1\geq 0,\ldots,\lambda_m\geq 0$ satisfying these conditions, then it is a solution of the optimization problem given by~\eqref{optim-1} and~\eqref{D_intersection}.
 \end{theorem}
 
\begin{proof} Let $\bar x \in D\cap (\dom f)$ be such that $\bar f(\bar x)=f(\bar x)$ and $\bar g_i(\bar x)=g_i(\bar x)$ for all $i\in I$, where~$I$ is the same as in the proof of Theorem~\ref{Lagrange_associated}. To derive the desired results from Theorems~\ref{Lagrange_associated}, we need some auxiliary facts, which are to be established now.
	
 For the given point $\bar x$, we have 
 \begin{equation}
 	\label{sudiff_n1}
	\partial f(\bar x)=\partial \bar f(\bar x),
\end{equation} 
\begin{equation}\label{sudiff_gi} 
\partial g_i(\bar x)=\partial \bar g_i(\bar x)
\end{equation} for all  $i\in I$, and 
 \begin{equation}\label{normal_n1}
N(\bar x;\Omega_0)=N(\bar x;\overline{\Omega_0}).
\end{equation} 
To prove the equality~\eqref{sudiff_n1}, take any $x^*\in\partial \bar f(\bar x)$ and have $\langle x^*,x-\bar x\rangle\leq \bar f(x)-\bar f(\bar x)$ for all $x\in\R^n$. Then, by the condition $\bar f(\bar x)=f(\bar x)$ and the inequality $\bar f(x)\leq f(x)$, we get $\langle x^*,x-\bar x\rangle\leq f(x)-f(\bar x)$ for all $x\in\R^n$. So, $x^*\in\partial f(\bar x)$. We have thus proved that $\partial \bar f(\bar x)\subset\partial f(\bar x)$. Next, fix any element $x^*\in\partial f(\bar x)$ and have  
	 \begin{equation}\label{y}\langle x^*,y-\bar x\rangle\leq f(y)-f(\bar x)
\end{equation} for all $y\in\R^n$. For any $x\in\R^n$, by taking $\liminf$  on both sides of the inequality~\eqref{y} as $y$ tends to $x$ and using Lemma~\ref{lem:bar_f_lsc}, we get $\langle x^*,x-\bar x\rangle\leq \bar f(x)- f(\bar x)$. As $\bar f(\bar x)=f(\bar x)$, from this we can deduce that $x^*\in\partial \bar f(\bar x)$. Therefore, $\partial f(\bar x)\subset\partial \bar f(\bar x)$. So, the equality~\eqref{sudiff_n1} is valid.
	
	For each $i\in I$, the proof of the equality~\eqref{sudiff_gi} is similar to that of the equality~\eqref{sudiff_n1}.
	
	Since $\Omega_0\subset \overline{\Omega_0}$, the inclusion $N(\bar x;\overline{\Omega_0})\subset N(\bar x;\Omega_0)$ is obvious. Now, take any $x^*\in N(\bar x;\Omega_0)$ and have $\langle x^*,y-\bar x\rangle\leq 0$ for all $y\in\Omega_0$. For each $x\in \overline{\Omega_0}$, there is a sequence $\{y_k\}\subset\Omega_0$ such that $\lim\limits_{k\to\infty} y_k=x$. Then, $\langle x^*,y_k-\bar x\rangle\leq 0$ for all $k\in\N$. Passing the last inequality to the limit as $k\to\infty$ gives $\langle x^*,x-\bar x\rangle\leq 0$. As the last inequality is valid for all $x\in \overline{\Omega_0}$, we have $x^*\in N(\bar x;\overline{\Omega_0})$.  We have thus shown that the equality~\eqref{normal_n1} holds.
	
 Now, we are in a position to prove the first assertion of the theorem. Suppose that $\bar x\in {\mathcal S}$, where ${\mathcal S}$ is the solution set of the optimization problem given by~\eqref{optim-1} and~\eqref{D_intersection}. Since the problem satisfies the generalized Slater condition in the sense of Definition~\ref{Slater}, where $x^0\in{\rm ri}(\dom f)$, the regularity condition~\eqref{reg1} is fulfilled (see the proof of Theorem~\ref{Lagrange_associated}). Hence, by Lemma~\ref{Sols} we have $\bar x\in {\mathcal S}_1$, where ${\mathcal S}_1$ is the solution set of the associated convex optimization problem~\eqref{optim-2}. So, applying Theorem~\ref{Lagrange_associated}, we can find such Lagrange multipliers $\lambda_1\geq 0,\ldots,\lambda_m\geq 0$ that the inclusion~\eqref{incl_bar} holds and $\lambda_i\bar g_i(\bar x)=0$ for $i\in I$. As $\bar g_i(\bar x)=g_i(\bar x)$ for all $i\in I$, the latter implies that $\lambda_ig_i(\bar x)=0$ for $i\in I$. Thanks to the relations~\eqref{sudiff_n1}--\eqref{normal_n1}, we easily obtain~\eqref{incl_original} from~\eqref{incl_bar}.
 
To prove the second assertion of the theorem, suppose that $\bar x\in {\rm ri}(\dom f)$ and there exist Lagrange multipliers $\lambda_1\geq 0,\ldots,\lambda_m\geq 0$ satisfying~\eqref{incl_original} and the condition $\lambda_ig_i(\bar x)=0$ for $i\in I$. Then, by the assumptions made and by~\eqref{sudiff_n1}--\eqref{normal_n1} we can get the inclusion~\eqref{incl_bar} and also the equality $\lambda_i\bar g_i(\bar x)=0$ for every $i\in I$. Hence, by Theorem~\ref{Lagrange_associated} we can infer that $\bar x\in {\mathcal S}_1$. Furthermore, Lemma~\ref{eq_in_ridom} assures that ${\rm ri} (\dom \bar f)={\rm ri}(\dom f)$. So, we have $\bar x\in \mathcal{S}_1\cap D \cap {\rm ri}(\dom \bar f)$. It remains to apply Corollary~\ref{Sols_1} to have $\bar x\in {\mathcal S}$. 
  
	The proof is complete.	$\hfill\Box$
\end{proof}
 
 \begin{remark}\label{rem4.1} As shown in the proof of Theorem~\ref{Lagrange_NC}, the inclusion~\eqref {incl_original} combined with the condition $\lambda_i g_i(\bar x)=0$ for $i\in I$ is equivalent to~\eqref{incl_bar} in combination with the requirement $\lambda_i\bar g_i(\bar x)=0$ for $i\in I$, provided that all assumptions of the theorem are satisfied.
\end{remark}

\begin{example}
	Consider the optimization problem given by~\eqref{optim-1} and~\eqref{D_intersection}, where
	$f(x)=x_1^2-x_1$ for $x=(x_1,x_2)\in\R^2,$
	$$\Omega_0=[0,1]\times[0,1]
	\setminus\left(\{0\}\times\Big(\dfrac{1}{4},\dfrac{2}{3}\Big)\right)
	\setminus\left(\Big(\dfrac{1}{4},\dfrac{2}{3}\Big)\times\{0\}\right),$$
	and $\Omega_1=\big\{x=(x_1,x_2)\in\R^2\mid g(x)\leq 0\big\},$
	where
	$g:\R^2\to\R$ is defined by
	$$g(x)=x_1^2-3x_1+x_2\quad\mbox{\rm for}\ x=(x_1,x_2)\in\R^2.$$
	To solve this problem by using Theorems~\ref{Lagrange_associated} and~\ref{Lagrange_NC}, we first observe that
	$\Omega_0$ is a nearly convex set with $\overline{\Omega_0}=[0,1]\times[0,1]$, $f$ and
	$g$ are continuous convex functions. Consequently, \begin{equation*}\label{normal_cones} N(x;\overline{\Omega_0})=\begin{cases}
			\R_-\times \R_-\quad&\text{if } x=(0,0),\\
			\R_-\times \{0\}&\text{if } x\in \{0\}\times(0,1),\\
			\R_- \times \R_+&\text{if } x=(0,1),\\
			\{0\}\times \R_+ &\text{if } x\in(0,1)\times\{1\},\\
			\R_+ \times \R_+ &\text{if } x=(1,1),\\
			\R_+ \times \{0\}&\text{if } x\in\{1\}\times(0,1),\\
			\R_+\times \R_-&\text{if } x=(1,0),\\
			\{0\}\times \R_-&\text{if } x\in (0,1)\times\{0\},\\
			\{(0,0)\}&\text{if } x\in (0,1)\times(0,1),
	\end{cases}\end{equation*}
	$\bar f=f$ and $\bar g=g$. In addition, since $\dom f=\dom g=\R^2$, we have $$\ri(\dom f)={\rm int}(\dom f)=\R^2.$$ Moreover, for any point $x^0\in\ri\Omega_0={\rm int}\,\Omega_0\subset \ri(\dom f)$, the generalized Slater condition as stated in Definition~\ref{Slater} is satisfied. Therefore, by Theorem~\ref{Lagrange_associated}, a point $\bar x=(\bar x_1,\bar x_2)$ from $\overline{\Omega_0}$ is a solution
	of the associated convex optimization~\eqref{optim-2} if and only if there exists
	a Lagrange multiplier $\lambda\ge0$ such that~\eqref{incl_bar}
	holds and the complementarity slackness condition
	$\lambda(\bar x_1^2-3\bar x_1+\bar x_2)=0$
	is fulfilled. Since
	$$\partial f(\bar x)=\{\nabla f(\bar x)\}= \{(2\bar x_1-1,0)\},\ \; 
	\partial g(\bar x)=\{\nabla g(\bar x)\}= \{(2\bar x_1-3,1)\},$$ the inclusion~\eqref{incl_bar} can be rewritten as 
	\begin{equation}\label{KKT_sys}
		\big(2(-\lambda-1)\bar x_1+(3\lambda+1),-\lambda\big)\in N(\bar x;\overline{\Omega_0}).
	\end{equation} 
	
 If $\lambda=0$, then by~\eqref{KKT_sys} we have \begin{equation}\label{KKT_c1}\big(-2\bar x_1+1,0\big)\in N(\bar x;\overline{\Omega_0}).\end{equation} 
	
	If $\big(-2\bar x_1+1,0\big)=(0,0)$, that is $\bar x_1=\dfrac{1}{2}$, then~\eqref{KKT_c1} is fulfilled. Besides, since every point $\bar x\in \Big\{\dfrac{1}{2}\Big\}\times [0,1]$ satisfies the constraint $g(\bar x)\leq 0$, we have $\Big\{\dfrac{1}{2}\Big\}\times [0,1]\subset {\mathcal S}_1$. In particular, for $\bar x=\Big(\dfrac{1}{2},0\Big)$, one has $\bar f(\bar x)=f(\bar x)=\bar x_1^2-\bar x_1=-\dfrac{1}{4}$. Thus, the optimal value of~\eqref{optim-2} is $\bar v=-\dfrac{1}{4}$. Combining this with~\eqref{cl_omega_1} and~\eqref{barD_cap} gives
	$$\begin{array}{rcl}
	{\mathcal S}_1 & = &\left\{\bar x\in \bar D\mid \bar f(\bar x)=\bar v\right\}\\
	  & = &\big\{\bar x\in \overline{\Omega_0}\cap\overline{\Omega_1}\mid \bar f(\bar x)=\bar v\big\}\\
	   & = &\big\{\bar x\in \overline{\Omega_0}\mid \bar g(\bar x)\leq 0,\ \bar f(\bar x)=\bar v\big\}\\
	    & = &\Big\{\bar x\in [0,1]\times [0,1]\mid g(\bar x)\leq 0,\ f(\bar x)=-\dfrac{1}{4}\Big\} \\
	    & = &\Big\{\bar x\in [0,1]\times [0,1]\mid g(\bar x)\leq 0,\ \bar x_1=\dfrac{1}{2}\Big\}.
	\end{array}$$ It follows that ${\mathcal S}_1=\Big\{\dfrac{1}{2}\Big\}\times [0,1].$ Using this result, Theorem~\ref{Lagrange_NC}, and Remark~\ref{rem4.1}, we find that ${\mathcal S}=\Big\{\dfrac{1}{2}\Big\}\times (0,1].$ 
\end{example}

 \section{Concluding Remarks}\label{sect5} 
 
Via the concept of associated convex optimization problem, we have obtained various properties of nearly convex optimization problems under geometrical constraints and functional constraints. Optimality conditions in the forms of Fermat's rules and Lagrange multiplier rules have been established. Several illustrative examples have been constructed.
 
 \section*{Acknowledgements}
 
 This research was supported by the project NCXS02.01/24-25 of Vietnam Academy of Science and Technology.  
 
 \section*{Declarations}
 
 \textbf{Conflict of interest} The authors have not disclosed any conflict of interest.
 
 \smallskip
 \noindent \textbf{Data availability statement} This manuscript has no associated data.


\begin{thebibliography}{99}
 	
  	\bibitem{at_2003} Auslender, A.,  Teboulle, M.: Asymptotic Cones and Functions in Optimization and Variational Inequalities, Springer-Verlag, New York (2003)
 	
 	\bibitem{nr1}  Bauschke, H.H.,  Hare, W.L.,  Moursi, W.M.:  On the range of the Douglas-Rachford operator, Math. Oper. Res. \textbf{41}, 884--897 (2016)
 	
 	\bibitem{bmw2013}  Bauschke, H.H.,  Moffat, S.M., Wang, X.: Near equality, near convexity, sums of maximally monotone operators, and averages of firmly nonexpansive mappings, Math. Program. 139, Ser. B, 55--70 (2013)
 	
 	 \bibitem{nr2}   Bauschke, H.H.,   Moursi, W.M.:  On the behavior of the Douglas-Rachford algorithm for minimizing a convex function subject to a linear constraint, SIAM J. Optim.  30, 2559--2576 (2020)
 	 	
  	\bibitem{bgw2007} 	Bo{\c{t}}, R.I., Grad, S.M., Wanka, G.: Almost convex functions: conjugacy and duality. In: Konnov, I., Luc, D.T., Rubinov, A. (eds.) Generalized Convexity and Related Topics, Lecture Notes in Economics 583, pp. 101--114. Springer-Verlag, Berlin Heidelberg (2007)
 	
 	\bibitem{bkw2008}  Bo{\c{t}}, R.I.,  Kassay, G., Wanka, G.: Duality for almost convex optimization problems via the perturbation approach, J. Global Optim. 42, 385--399 (2008)
 	
 	\bibitem{gm2021}  Ghafari, N., Mohebi, H.: Optimality conditions for nonconvex problems over nearly convex feasible sets, Arab. J. Math. 10, 395--408 (2021)
 	
 	\bibitem{nr3}  Hinrichsen, D., Oeljeklaus, E.: The set of controllable multi-input systems is generically convex, Math. Control. Signal Syst. 31 , 265--278 (2019)
 	
 	\bibitem{ho1}  Ho, Q.: Necessary and sufficient KKT optimality conditions in non-convex optimization, Optim. Lett. 11, 41--46 (2017)
 	
 	\bibitem{IT_1979} Ioffe, A.D.; Tihomirov, V.M.: Theory of Extremal Problems,
 	North-Holland Publishing Co., Amsterdam-New York (1979)
 	
 	\bibitem{jm1}  Jeyakumar, V.,  Mohebi, H.: Characterizing best approximation from a convex set without convex representation,
 	J. Approx. Theory 239, 113--127 (2019)
 	
 	\bibitem{LM2019}   Li, J., Mastroeni, G.: Near equality and almost convexity of functions with applications to optimization and error bounds, J. Convex Anal.  26, 785--822 (2019)
 	
 	\bibitem{nr5}  Luo, H., Wang, X., Lukens, B.: Variational analysis on the signed distance functions, J. Optim. Theory Appl. 180, 751--774 (2019)
 	
 	\bibitem{Minty1961}  Minty, G.J.: On the maximal domain of a ``monotone'' function,  Michigan Math. J. 8, 135--137 (1961)
 	
 	\bibitem{mmw2016}  Moffat, S.M.,  Moursi, W.M.,  Wang, X.:
 	Nearly convex sets: fine properties and domains or ranges of subdifferentials of convex functions,
 	Math. Program. 160, Ser. A, 193--223 (2016)
 	
 	
 	\bibitem{nr4}  Moursi, W.M.: The forward-backward algorithm and the normal problem,  J.
 		Optim. Theory Appl. 176, 605--624 (2018)
 	
 	\bibitem{nty1}   Nam, N.M.,  Thieu, N.N.,  Yen, N.D.: Near convexity and generalized differentiation, J. Convex Anal. 32, 605--630 (2025)
 	
 	\bibitem{r}  Rockafellar, R.T.: Convex Analysis, Princeton University Press, Princeton, New Jersey (1970)
 	
 	\bibitem{R1970} Rockafellar, R.T.: On the virtual convexity of the domain and range of a nonlinear maximal monotone operator, Math. Ann. 185, 81--90 (1970)
 		
 \end{thebibliography}
\end{document}